\documentclass[11pt]{amsart} 
\usepackage{amsmath, amssymb, latexsym, tikz}
   \usepackage{forest}
\usepackage{caption,ifthen,url,fancyhdr,cite}
\usepackage[foot]{amsaddr}          
\usepackage{libertine}      
\usepackage{longtable}    
\usepackage{upquote}
\usepackage{listings}   
\usepackage{textcomp}    
\usepackage{adjustbox}
\usepackage{booktabs} 
   \usepackage{etoolbox}
    \usepackage[all]{xy}
    
  \usepackage[pdftex,colorlinks=true, citecolor=blue, linkcolor=blue, urlcolor=black]{hyperref}

\makeatletter
\patchcmd{\@verbatim}  
  {\verbatim@font}
  {\verbatim@font\scriptsize}
  {}{}
  \makeatother  
  
\newcommand\textcode[1]{{\footnotesize \rm\texttt{#1}}}   

\makeatletter 
\renewcommand\subsubsection{\@startsection{subsubsection}{3}%
  \z@{.5\linespacing\@plus.7\linespacing}{-.5em}%
  {\normalfont\bfseries}} 
\makeatother
 
\usepackage[T1]{fontenc}

\newcommand\makequotestraight{%
\begingroup\lccode`~=`' 
\lowercase{\endgroup\let~}\textquotesingle
\catcode`'=\active
}

\usepackage[a4paper,right=2.5cm, left=2.5cm, top=2.5cm, bottom=2.5cm, marginpar=1.6cm]{geometry}

\newtheoremstyle{mythm}                   
{6pt}
{6pt}
{\it}
{}
{\bf}
{.}
{.5em}
{}

\newtheoremstyle{mydef}                   
{6pt}
{6pt}
{}
{}
{\bf}
{.}
{.5em}
{}

\newtheoremstyle{myrem}                   
{6pt}
{6pt}
{}
{}
{\bf}
{.}
{.5em}
{}

\theoremstyle{mythm}      
\newtheorem{theorem}{Theorem}[section]
\newtheorem{proposition}[theorem]{Proposition}
\newtheorem{lemma}[theorem]{Lemma}
        
\newtheorem{conjecture}[theorem]{Conjecture} 

\theoremstyle{mydef}      
\newtheorem{definition}[theorem]{Definition}
\newtheorem{example}[theorem]{Example}

\theoremstyle{myrem}
\newtheorem{remark}[theorem]{Remark}
\numberwithin{equation}{section}

\newcounter{ithmcount}

\renewcommand{\leq}{\leqslant} 
\renewcommand{\geq}{\geqslant}

\newcommand{\Aut}{{\rm Aut}}

\newcommand{\la}{\langle}
\newcommand{\ra}{\rangle}
\newcommand{\sss}{\sigma}
\newcommand{\aaa}{\alpha}
\newcommand{\ZZ}{{\mathbb{Z}}}
\newcommand{\bc}{\begin{center}}
\newcommand{\ec}{\end{center}}
\newcommand{\ol}{\overline}
\newcommand{\bi}{\begin{itemize}}
\newcommand{\ei}{\end{itemize}}

\makeatletter
\newcommand{\udot}{\mathpalette\udot@\relax}
\newcommand{\udot@}[2]{%
  \begingroup
  \sbox\z@{$#1{:}$}%
  \sbox\tw@{$#1{.}$}%
  \raisebox{\dimexpr\ht\z@-\ht\tw@}{$\m@th#1.$}%
  \endgroup
}
\makeatother

\newcommand{\Z}{\mathbb{Z}}

\usepackage{fancyvrb}
\fvset{commandchars=\\\{\}}

\makeatletter
\@namedef{subjclassname@1991}{2020 MSC}
\makeatother

\begin{document}

\title{On the trivial units property and the unique product property}
\subjclass[2000]{} 
\author[H. Dietrich]{Heiko Dietrich}
\author[M. Lee]{Melissa Lee}
\author[A. Nies]{Andr\'e Nies}
\author[M. Vinyals]{Marc Vinyals}
\address[Dietrich, Lee]{School of Mathematics, Monash University, Clayton VIC 3800, Australia}
\address[Nies, Vinyals]{School of Computer Science, University of Auckland, Auckland 1142, New Zealand} 
\email{\rm heiko.dietrich@monash.edu, melissa.lee@monash.edu, andre@cs.auckland.ac.nz, marc.vinyals@auckland.ac.nz}
\thanks{The second author acknowledges the support of an Australian Research Council Discovery Early Career Researcher Award (project number DE230100579) The third and fourth authors acknowledge the support of the Marsden fund of New Zealand. The first author thanks the University of Auckland for the hospitality and support for various research stays.}
\date{\today}

\begin{abstract}
  We report on some computational experiments related to the trivial units property and unique product property for group rings of torsion-free groups. These properties are  related to   Kaplansky's unit and zero-divisor conjectures. Our investigations include a classification of certain symmetric non-trivial units in the binary group ring of the Hantzsche-Wendt group; this group was used in Gardam's refutal of Kaplansky's unit conjecture. We also exhibit and investigate a new candidate group that fails the unique units property but may satisfy the trivial unit property.  No examples of groups with these properties are known to date.
\end{abstract} 

\maketitle
\section{Introduction}\label{sec_intro} 
\noindent Let $L$ be a ring with identity, not necessarily commutative. A \emph{zero-divisor} in $L$ is a non-zero $x\in L$ such that $xy=0$ or $yx=0$ for some nonzero $y\in L$. 
One says that $x\in L$ is a \emph{unit}  if there exists $y\in L$ such that $xy=1$ and $yx=1$; if $L$ has no zero-divisors, then one of the conditions implies the other. We will be interested in the group ring $R[G]$ of a group $G$ and commutative ring $R$ with identity. Our paper focusses on two interrelated conditions: the trivial units property for group rings, and the unique product property for groups. We begin by introducing these conditions and giving brief background. 
\begin{definition} \label{df:TUP} A   group $G$ satisfies the \emph{trival units   property (TUP) for} an integral domain~$R$ if the only units in the group ring $R[G]$ are of the form $rg$, where $g\in G$ and $r $ is a unit of $R$. Such  units will be called  trivial.  \end{definition} 
The following long-open conjecture was refuted by Gardam~\cite{G}.
\begin{conjecture} \label{UC}
	If $G$ is a torsion-free group and $R$ an integral domain, then $G$ satisfies the TUP for $R$. 
\end{conjecture}
Gardam \cite{G} found   a non-trivial unit in $\mathbb{F}_2[P]$, where $\mathbb{F}_2$ is the field with two elements and $P$ is the torsion-free group defined by
\begin{align}\label{eq_P}P &= \langle a,b \mid b^{-1}a^2b = a^{-2}, a^{-1}b^{2}a=b^{-2}\rangle.
\end{align}
 We note that $P$ is a crystallographic group, and $x=a^2, y=b^2$ and $z=abab$ (in the notation of \cite{G}) are   free generators of the abelian translation subgroup of $P$, which has index 4 in $P$. 
\begin{remark} \label{rem:Gardam SAT}
	While the proof given in Gardam's paper is purely mathematical, the original discovery of a non-trivial unit in $\mathbb{F}_2[P]$ was made computationally using a SAT solver.   Gardam found   a  non-trivial unit supported on  a ball of radius $5$  around the identity in the Cayley graph of $P$. To do so, he  reformulated the problem of finding a unit as a Boolean satisfiability (SAT) problem, assigning Boolean variables for the coefficients in $\mathbb F_2$  to    words in the generators of $P$ and their inverses. 
\end{remark}

The  following purely  group-theoretic property was introduced by  Rudin and Schneider~\cite{Rudin.Schneider:64}, who used the term $\Omega$-group. 
\begin{definition} \label{df:UPP}	A group $G$ has the  \emph{unique product property} (UPP)   if for all non-empty subsets $A,B\subseteq G$,  there exist some $a\in A$ and $b \in B$ such that whenever $ab=a'b'$ with $a'\in A$, $b'\in B$, then $a=a'$ and $b=b'$. In other words, when viewing $AB$ as a multiset, some element occurs  with multiplicity $1$. 
\end{definition}

 Promislow~\cite{P} showed that $P$ fails the UPP. It is well known that UPP   implies the TUP for every domain;~see Proposition~\ref{prop:chain} below for a proof.   Suprisingly, the converse implication is still unknown: is   there  a group that satisfies the TUP for some domain, but fails the UPP? This question is implicit in early work, and explicitly  asked for instance in \cite{G}.

This  paper follows an experimental approach,  using SAT solvers and computer algebra  towards   new examples, which are then evaluated for advancing the  theory. We 
\bi 
\item describe  nontrivial  units in $\mathbb{F}_2[P]$ supported on   balls from radius   4 onwards, and their potential symmetry properties;
\item expose a new candidate group  that fails the UPP, but may   satisfy the TUP for some domain.  This group,  denoted $H_4$, is a small index extension of the usual integral Heisenberg group.
\ei

\section{Mathematical background}

\subsection{Group rings}   For a group $G$ and a commutative ring $R$ with multiplicative identity, the \textit{group ring} $R[G]$ consists of all the  formal sums $u=\sum_{g\in G} u_gg$ where only finitely many coefficients $u_g\in R$ are non-zero. The addition operation  is carried out component-wise. The operation  of multiplication is given by  \begin{center} $(\sum_{g\in G} u_gg )(\sum_{h\in G} v_hh )= \sum_{r \in G} w_r r$,\end{center} where $w_r = \sum_{gh = r} u_gv_h$. By slight abuse of notation, we usually write $1=1_G$ and $1=1_R$. The \emph{support} of $\sum_{g\in G}u_g g$ is the subset $\{g\in G \mid u_g\ne 0\}$.

\subsection{Some background on the unit conjecture}  
Higman posed Conjecture~\ref{UC}     in his unpublished 1940 PhD thesis \cite[p.~77]{H1}  for the case that $R = \ZZ$. The conjecture was  taken up  in generality  by Kaplansky \cite{K1}, and  became known as the   \textit{unit conjecture}.  In his paper, Kaplansky attributes it to a list of problems arising from a 1968 conference in Moldova.  Note that if  $R[G]$  has a non-trivial unit, then after extending  the integral domain~$R$ to  its field of fractions~$K$, one  obtains a   counterexample   in   the group ring $K[G]$.   Kaplansky also posed the stronger  {\it zero divisor conjecture:} $K[G]$ has no zero divisors, and the even stronger {\it idempotent conjecture}, that  the only idempotents in $K[G]$ are $0,1$.

There are various families   of groups that have the TUP for each integral domain $R$.   Higman   showed this    for abelian groups, and more generally for \textit{locally indicable groups}; this means that    every non-trivial finitely generated subgroup has a quotient isomorphic to $\mathbb{Z}$.

After Gardam's refutation~\cite{G}, his work was generalised by Murray \cite{murray}, who altered the  construction to obtain    non-trivial units in $\mathbb{F}_q[P]$ for every prime $q$.   Gardam \cite{G2}   found   a non-trivial unit in $\mathbb{C}[P]$; it has   the same support as the original example over $\mathbb{F}_2[P]$. Thus, the conjecture is refuted for fields of each characteristics. However, the original conjecture of Higman was   for integral group rings, and   remains open.

\subsection{Background on the unique product property}

The study of the UPP defined in Definition \ref{df:UPP} has a long history, and    was instrumental to   the refutation of the unit conjecture \cite{G}.
On the positive side, every bi-orderable group has UPP because the product of the maxima of two sets is unique; in fact, the weaker property of being diffuse suffices~\cite{bowditch}; for definitions and a proof see Proposition \ref{prop:chain} below.

 The first example of a torsion-free group that fails  the UPP was constructed by Rips and Segev \cite{RS} who made use of small cancellation theory. Passman introduced the torsion-free group $P$ of \eqref{eq_P} in connection with the Kaplansky conjectures.  Promislow \cite{P} exhibited a $14$-element subset $A\subseteq P$ such that the multiset $AA$ has no element with multiplicity~$1$, showing that $P$ does not satisfy the UPP. Carter \cite{C} generalised $P$ to an infinite family of torsion-free groups
\begin{align}\label{eq_Pk} 
	P_k &= \langle a,b \mid a^{-1}b^{2^k}a = b^{-2^k}, b^{-1}a^2b=a^{-2} \rangle\quad(k>0)
\end{align}
and showed that the $P_k$ are pairwise non-isomorphic, do not contain $P$ as a subgroup for $k>1$, and none satisfies the UPP; note that $P=P_1$. Craig and Linnell \cite{CL} conjectured that every uniform pro-$p$-group satisfies the UPP. They also consider some variants of $P$ with more than two generators and show that each is torsion-free. None of these variants satisfies the UPP since they each contain $P$ as a subgroup. Lastly, we mention that Nielsen and Soelberg proved bound for the sizes of sets in torsion-free groups that witness the non-unique product property; for example, they proved that if $AA$ has a non-unique product, then $|A|\geq 8$ and this bound is sharp, see \cite[Theorem~1.2]{nielsen}.
  We refer to \cite{CL} and the references therein for more details on the unique product property.

  As a quick observation, the  UPP   is related to  zero divisors in $K[G]$ as follows: Let $u,v\in K[G]$ be non-zero elements with support $A$ and $B$, respectively. Let $C$ be the set of elements in the multiset $AB$, so that 
  \[uv=\sum_{c\in C} \big(\sum_{a\in A, b\in B\atop ab=c} u_av_b\big) c.\]
  If $uv=0$, then this requires that every $c\in C$ occurs with multiplicity at least $2$ in $AB$; in particular, if $K[G]$ has zero divisors, then $G$ does not satisfy the UPP.

 \subsection{Implications between conditions}

We  describe a hierarchy of properties for $K[G]$, from\linebreak strongest to weakest,  which includes the properties on which  the three Kaplansky conjectures are based. While the  implications seem well-known (see also \cite{bowditch,blog}),   we include new short proofs and references for completeness. Recall that a group $G$ is \emph{left-orderable} if there is a total order ``$\leq$'' on $G$ that is invariant under left multiplication. A group  $G$ is \emph{diffuse} \cite{bowditch} if for every non-empty finite subset $C\subseteq G$ there exists an \emph{extremal element} in $C$, that is, $c\in C$ such that for all non-identity $g\in G$, either $gc \notin C$ or $g^{-1} c\notin C$.

\begin{proposition}
\label{prop:chain}
Let $K$ be a field and let  $G$ be a torsion-free group. Then $(1)\Rightarrow (2)\Rightarrow \ldots\Rightarrow(5)$, where $(1),\ldots,(5)$ are the following properties:
\begin{enumerate}
\item \label{lo} $G$ is left-orderable.
\item \label{diff} $G$ is diffuse. 
\item \label{upp} $G$ satisfies the UPP. 
\item \label{uc} $K[G]$ has only the trivial units.
\item \label{zdc} $K[G]$ has no zero divisors.
\end{enumerate}					
\end{proposition}
\begin{proof} $\eqref{lo}\Rightarrow \eqref{diff}$: Let $C\subseteq G$ be finite and let $c\in C$ such that  $b\leq c$ for all $b\in C$. If $g\in G$ is non-trivial, then either $gc>c$ (and then $gc\notin C$ by the maximality of $c$), or $gc<c$ (and then $c<g^{-1}c$, which implies that $g^{-1}c\notin C$). Thus, $G$ is diffuse.

\noindent$\eqref{diff}\Rightarrow \eqref{upp}$: Let $A,B\subseteq G$ be finite subsets and let $c$ be an extremal element in the \emph{set} $C=AB$. Suppose $c=ab=a'b'$ with $a,a'\in A$ and $b,b'\in B$ such that $g=a'a^{-1}\ne 1$. Then $gab=a'b\in C$ and $g^{-1}ab=a(a')^{-1}ab=ab'\in C$, which contradicts $\eqref{diff}$. Thus, the extremal element in   $AB$ viewed as a  \emph{multiset} has multiplicity $1$, that is, $G$ satisfies the UPP.

   \noindent $\eqref{upp}\Rightarrow \eqref{uc}$: We prove the contrapositive.  Suppose $A,B\subseteq G$ are   the support sets for a non-trivial  unit $\aaa$ in $K[G]$ and its inverse $\beta$, respectively. By assumption, there exists $x\in A$ such that  $x^{-1} \in B$. By replacing $\aaa$ by $x^{-1}\aaa$ and  $\beta$ by $\beta x$, we can assume that $1\in A\cap B$. 
   
   Let $E=B^{-1}A$ and $F=BA^{-1}$; we show that $E$ and $F$ do not have the unique product propery. Since $|A|,|B|\geq 2$,  there exists a   $a\in A-\{1\}$, which shows that the product $1=1\cdot 1=a\cdot a^{-1}$ is not unique in $EF$. Now consider a non-identity product $x=( b_0^{-1} a_0 )(b_1 a_1^{-1}) $ in $EF$ where $a_0,a_1 \in A$ and $b_0,b_1\in B$. If $a_0=b_1=1$, then $x=b_0^{-1}a_1^{-1}=b_0^{-1}aa^{-1}a_1^{-1}$ is not unique. If $a_0b_1=1$ and $a_0\ne 1$, then $x= b_0^{-1}a_0b_1a_1^{-1}=b_0^{-1}a_1^{-1}$. Lastly, consider $a_0b_1\ne 1$. Since $A$ and $B$ are supports for a unit, there must exist $a_2\in A$ and $b_2\in B$ with $a_2\ne a_0$, $b_2\ne b_1$, and $a_0b_1=a_2b_2$; thus,  $ (b_0^{-1} a_0) (b_1 a_1^{-1}) =  (b_0^{-1} a_2) (b_2 a_1^{-1})$ is not unique. The claim follows.

   \noindent$\eqref{uc}\Rightarrow \eqref{zdc}$: Since $G$ is torsion-free, the only finite normal subgroup is the identity subgroup. Hence the ring $K[G]$ is \emph{prime} (see \cite[Theorem 4.2.10]{passman}), that is, if $IJ=0$ for  ideals $I,J $ in $K[G]$, then $I=0$ or $J=0$. Suppose, for a contradiction, that $u,v\in K[G]$ are zero divisors, and denote by $I$ and $J$ the principal ideals generated by $u$ and by $v$, respectively. Since $I$ and $J$ are non-zero, the prime property shows that $IJ\ne 0$. This, in turn, implies that there exists $r\in K[G]$ such that  $vru\ne 0$.  The element $t=vru$ therefore is non-zero and satisfies $t^2=0$ and $(1+t)(1-t)=(1-t)(1+t)=1$. Thus, $1-t$ is a unit in $K[G]$. By assumption, this unit must be trivial, say $1-t=kg$ for some $g\in G$ and $k\in K$. Now  $t=1-kg\in K[\langle g\rangle]$ is a non-zero element with $t^2=0$. Since $\langle g\rangle$ is infinite cyclic, the ring $K[\langle g\rangle]$ is isomorphic to the ring of univariate Laurent polynomials over $K$; in particular, it is an integral domain. But then $t^2=0$ forces $t=0$, a contradiction. Thus, $K[G]$ has no zero divisors.
\end{proof}

  Gardam's work \cite{G} shows that $\mathbb{F}_2[P]$ with $P$ as in \eqref{eq_P} is an integral domain that contains non-trivial units; this shows that ``\eqref{zdc} $\not\Rightarrow$(\ref{uc})'' in general. An example showing ``\eqref{diff} $\not\Rightarrow$\eqref{lo}'' is given by Raimbault, Kionke, and  Dunfield \cite{RKD}. It is an interesting open problem to find examples of the other non-implications. As mentioned, our experimental work aims at  an example of a torsion-free group that satisfies \eqref{uc} for some domain  $R$, but not \eqref{upp}.

Cliff \cite[Theorem 2]{cliff} showed that if $K$ is a field of characteristic $p>0$ and $G$ is a torsion-free polycyclic-by-finite group, then $K[G]$ has no zero divisors.


\section{Experimental results on units in $P$}
\subsection{The group $P$}   Recall from (\ref{eq_P}) that $P$ is the group given by the presentation
\begin{equation*}P = \langle a,b \mid b^{-1}a^2b = a^{-2}, a^{-1}b^{2}a=b^{-2}\rangle.
\end{equation*}
This presentation yields a complete set of rewrite rules: for  $\gamma= \pm 1$, $\delta= \pm 1$, $\varepsilon = \pm 2$,
 \begin{equation} \label{eqn:rewrite} a^{\varepsilon}b^\delta \leftrightarrow b^\delta a^{-\varepsilon} \qquad b^{\varepsilon}a^\delta \leftrightarrow a^\delta b^{-\varepsilon} \qquad a^\delta b^\gamma a^\delta \leftrightarrow a^{-\delta} b^\gamma a^{-\delta}\qquad b^\delta a^\gamma b^\delta \leftrightarrow b^{-\delta} a^\gamma b^{-\delta}. \end{equation}
 To give an idea of the size of search spaces, we  display  the   sizes for  the balls  and spheres  around $1$ in the Cayley graph of $P$, depending on the radius up to 6.  

\begin{center}
	\begin{tabular}{c|ccccccc}
		 & 0 & 1 & 2 & 3 & 4 & 5 & 6 \\
		\hline
		Radius      & 0 & 1 & 2 & 3 & 4 & 5 & 6 \\
		Ball size   & 1 & 5 & 17 & 41 & 83 & 147 & 239 \\
		Sphere size & 1 & 4 & 12 & 24 & 42 & 64 & 92 \\
	\end{tabular}
\end{center}

\subsection{Two nontrivial units} We replicate Gardam's breakthrough result \cite{G} and report on a few examples of non-trivial units that we have computed with a modification of the method described in   Remark~\ref{rem:Gardam SAT}.

\begin{example}\label{ex1}
  With \textcode{Kissat} \cite{kissat}, we were able to find a non-trivial unit whose support lies in a ball of radius $6$ in $P$. We used the $2$-generator description of~$P$ as a subgroup of $(D_\infty)^3$ given in \cite{P}. Here $D_\infty$ is the infinite dihedral group; see also the code for GAP \cite{gap} provided in Figure \ref{fig_P}. Our computation took about $5$ minutes on a 2020 M1 MacBook~Air.

  Recall that we write  $x= a^2, y= b^2, z= (ab)^2$, and  $\la x,y,z \ra$ freely generate the maximal  abelian normal subgroup of $P$, which  has index $4$.  The unit can be decomposed as  $p + qa + rb +sab$ with inverse $p' + q'a +  r'b +s'ab$, where $p,q,r,s, p',q',r',s' \in \mathbb F_2[\ZZ^3]$ are as below.  We view  $x,y,z$ as standard generators of $\ZZ^3$, and  write $\ol x$ for $x^{-1}$ etc.
\begin{eqnarray*}
p &=& 1 +\ol x + \ol z +   \ol y + \ol x\ol y + \ol y\ol z + \ol x\ol z +  \ol x\ol y\ol z   \\
q &=&  1 + \ol x z + \ol x y z + x^{-2} y \\
r &=& \ol y +  \ol x z + \ol x\ol y +  y^{-2} z \\
s &=&   \ol x\ol z + \ol x y + y\ol z + x^{-2} y\ol z + \ol x y^2\ol z\\[1ex] 
p' &=& 1 + x + y + z +  x y   + x z  + y z  + x y z  \\
q' &=&  x + z + y z + \ol x y \\
r' &=& 1 + \ol x   +  \ol y z + \ol x y z \\ 
s' &=&  y + \ol x + \ol x y\ol z + \ol x y^2  + x^{-2} y
\end{eqnarray*}
This unit  $u$ is essentially different from Gardam's: For instance, $s$ has a term of length $4$, and no easy transformation such as $u \to ux, u \to u^a$ reduces the minimal length of all the terms below $4$.  In contrast, the Laurent polynomials in Gardam's unit and its inverse have  all terms of length at most 3. Curiously, $p $ and $p'$ are very similar to Gardam's, and also the support size distribution is the same as for Gardam's, $|p|= |p'|=8, |q| = |q'|= |r| = |r'| = 4, |s|= |s'| = 5$.
\end{example}

\begin{figure}[h]
\begin{verbatim}
## define polycyclic presentation for supergroup (D_\infty)^3 of P
   coll := FromTheLeftCollector(6);
   for i in [1,3,5] do SetRelativeOrder(coll,i,2); od;     
   SetConjugate(coll,2,1,[2,-1]);; 
   SetConjugate(coll,4,3,[4,-1]);;   
   SetConjugate(coll,6,5,[6,-1]);;
   D  := PcpGroupByCollector(coll);; 
## define P and group ring over GF(2)
   P  := Subgroup(D,[D.2*D.3*D.5,D.1*D.4*(D.6*D.5)]);;
   FG := GroupRing(GF(2),P);;        
   a  := P.1*One(FG);;               
   b  := P.2*One(FG);;
## get two non-trivial units
   U := Sum([ b*a^3, b*a^2*(b*a^-1)^2, b*a^2*b^-2, b*a*b^2,
              b*a*(a*b^-1)^2, b*a*b^-2, b^-1, b^-2*a^-1*b*a^-1,
              b*a^-1, a, (b*a^-1)^2, One(FG), b^-1*a^-1*b*a^-1,  
              b^-2, (a^-1*b)^2*a^-1, a^-1*b^-1*a^-1*b*a^-1,
              a^-1*(a^-1*b)^2*a^-1, a^-2, a^-2*b^-1*a^-1*b*a^-1,
              a^-2*b^-2, a^-3*b^-2 ]);;
   V := Sum([ b*a^3*(a*b^-1)^2, b*a^2*b^2*(b*a^-1)^2, b*a^2,
              b*a*b^2*(a*b^-1)^2, b*a, b*a*b^-1*(b^-1*a)^2*b^-1, b,
              a^-1*b*a^-1, a*b^-1, a^3, a^2*b^2, a^2*b^2*(a*b^-1)^2,
              a^2, a^2*(a*b^-1)^2, a*(b*a^-1)^2, a*b^-1*a^-1*b*a^-1, 
              b^2, b^2*(a*b^-1)^2, One(FG), (a*b^-1)^2, a^-1*b^-2 ]);;
## check that they multiply to the identity
   U*V = U^0;
## true
\end{verbatim}
\caption{GAP code for Example \ref{ex1} that demonstrates that $\mathbb{F}_2[P]$ has non-trivial units; the group $P$ is constructed following  \cite{P}.}\label{fig_P}
\end{figure}

\begin{example}Using the decomposition in Example \ref{ex1}, another   unit is $\mathbb{F}_2[P]$ is defined by  
\begin{eqnarray*}
p &=& y+ z+\ol xz+\ol x\ol y  \\
q &=&  1+\ol x\ol z+\ol xy+\ol x\ol y+ x^{-2}  \\
r &=&1+\ol y+ x+ x\ol y+\ol yz+ x\ol yz+ z, xz \\
s &=&   y\ol z+\ol x+\ol z+ xy\\[1ex]
p' &=& y+x\ol y+\ol z+ x\ol z  \\
q' &=&\ol x+ x+\ol y+ y+\ol z \\
r' &=&  1+\ol y+ x+ x\ol y+\ol yz+ x\ol yz+ z+ xz \\ 
s' &=&  1+ y+\ol xy\ol z+ x\ol z.
\end{eqnarray*}
\end{example}

 \subsection{Automorphisms of $P$ fixing $\{a, b, a^{-1}, b^{-1}\}$} \label{ss:automorphism P}  Note that  the assignment on generators $a \mapsto b$, $b \mapsto a$  extends to  an automorphism $\pi \colon P \to P$.  Also,  any  automorphism of a group $G$  extends naturally to a group ring $K[G]$. Below we will  discuss    nontrivial units $\alpha$  in $\mathbb F_2[P]$ with inverse $\pi(\alpha)$:
 \begin{definition}  \label{df:swap} For a field $K$,  a \textit{nontrivial} unit  $U$ in $K[P]$ is a \textit{swap unit}  if its inverse is obtained by applying the swap automorphism~$\pi$, that is, $U^{-1}= \pi(U)$.  \end{definition} Let $\alpha$ be the automorphism given by $a \mapsto a^{-1}, b \to b$, and  let $\beta $  be the automorphism of $P$ given by $a \mapsto a, b \to b^{-1}$. The group of automorphisms of $P$ that fix the set  $\{a, b, a^{-1}, b^{-1}\}$ is \begin{equation} \label{eqn:S} S=\la \alpha, \beta \ra \rtimes \la \pi \ra,\end{equation} 
  which has size $8$.

 \subsection{The  nontrivial units of   $\mathbb{F}_2[P]$   on a ball of radius~4}
\label{ss:rad4}

  Our computations showed that  there exist no nontrivial units whose support lies in a ball of radius $3$ around
  $1_P$  in the Cayley graph of $P$. There exist exactly 36 units
  whose support lies in a ball of radius $4$ around $1_P$.  They are displayed in Table~\ref{table:radius 4} in the appendix.   We note  that $V_1$,  the inverse of  $U_1$, is obtained by replacing $baba$ by its inverse $abab$ in $P$; the same holds for $V_2$ versus $U_2$.  
  
  We enumerated all such
  units by expressing the unit and its inverse as a satisfying assignment to a SAT formula, and then using \textcode{TabularAllSat}~\cite{SpallittaBiereSebastiani2025,TabularAllSATCode} to enumerate all satisfying assignments. The
  enumeration procedure took approximately~$2$ hours. The enumeration
  procedure on a ball of radius $5$ did not terminate within one day.  We remark that the choice of a description for $P$ significantly
  affects the behaviour of \textcode{Kissat}. While the search for units
  in a ball generated according to~\cite{P} terminates in a few minutes,
  the search in a ball generated according to~\cite{G} does not terminate
  within one day.

  We used GAP to determine the orbits under the group $S$ from (\ref{eqn:S}) among the units in radius $4$:  
  \[
  \begin{array}{l}
  	\{U_1, V_1, U_2, V_2\} \\
  	\{U_3, U_4, V_5, U_8, U_{10}, V_{13}, V_{14}, V_{18}\} \\
  	\{V_3, V_4, U_5, V_8, V_{10}, U_{13}, U_{14}, U_{18}\} \\
  	\{U_6, U_7, V_9, U_{11}, U_{12}, V_{15}, V_{16}, V_{17}\} \\
  	\{V_6, V_7, U_9, V_{11}, V_{12}, U_{15}, U_{16}, U_{17}\}
  \end{array}
  \]
  The first orbit is given by the facts that $U_1$, $U_2$ are swap units and that    $\alpha(U_1) = \beta(U_1)= U_2$ (see Subsection~\ref{ss:automorphism P} for the automorphisms  $\alpha $ and $\beta$ of $P$).  The first orbit is closed under inverses, the third is obtained by  inverting the elements  of the second, and the fifth by  inverting the elements  of the fourth.   So there are only three essentially different units (not counting inverses as separate), for instance represented by $U_1, U_3$ and $U_6$. We note that  even under the action of $\Aut(P)$, the unit $U_1$ is not in the same orbit as $U_3$ or $U_6$, because $U_1$ contains five squares, and $U_3$ and $U_6 $ only two.

    \subsection{Swap  units  of   $\mathbb{F}_2[P]$  on balls of radius~5 and 6}
 \label{ss:rad5and6}
The  units  $U_1,V_1,U_2,V_2$ in Table~\ref{table:radius 4} are the only swap units supported on the ball of  radius $4$ around $1_P$.   
We next searched for the swap units supported on the ball of \textit{radius $5$ around~$1_P$}, and determined that there are  exactly~20.  This computation took only a few  seconds, reflecting that predetermining the inverse roughly halves the number of variables in the SAT formula for which we find all satisfying assignments. 

  Being a  swap unit is not in general preserved by applying $\alpha$ or $\beta$. Rather, for a swap unit $U$, since $\pi \circ \beta= \alpha \circ \pi$, we have $\alpha(U)^{-1}= \alpha(\pi(U))= \pi (\beta(U))$. So   $\alpha(U)$ is a swap unit iff  $\alpha(U)= \beta(U)$ iff $\beta(U)$ is a swap unit. This holds for instance for $U_1$ as mentioned above.  Let $T$ be the four-element subgroup of $S$ generated by $\pi$ and $\alpha \circ \beta$. Since $\alpha \circ \beta$ commutes with $\pi$, any automorphism in $T$ preserves being a swap unit. Using GAP we determined that the  set of units supported on a ball of radius $5$ but not   radius~$4$ around $1_P$  is partioned into    four  $T$-orbits of length $4$. In particular, no such unit $W$ satisfies $\alpha(W) = \beta(W)$.  A choice of   representing units for these orbits is displayed in Table~\ref{table:radius 5}. 
All the   units  contain $1$ and   have exactly one element with shortest representation of length $5$. 
 \setlength{\fboxsep}{1pt}
\setlength{\fboxrule}{0.8pt}
\begin{table}[h]
	\caption{Four  swap units   on  a ball of radius 5 but not radius 4,  representing the four $T$-orbits of such  units.} 
	\label{table:radius 5}

	{\fontsize{11}{8.5}\selectfont
		\begin{center}
			\setlength{\tabcolsep}{3pt}
\begin{tabular}{|c|c|c|c|c|c|c|c|c|c|c|}
	\noalign{\hrule height .04cm}
	\fbox{$W_{1}$} & $1$ & $a$ & $ab$ & $b^{2}$ & $ba^{-1}$ & $a^{-2}$ & $a^{-1}b^{-1}$ & $a^{2}b^{-1}$ & $ab^{-1}a$ & $ab^{-1}a^{-1}$ \\
	\hline
	$bab$ & $bab^{-1}$ & $b^{3}$ & $a^{3}b^{-1}$ & $abab$ & $ab^{-1}ab$ & $ab^{-3}$ & $ba^{2}b$ & $a^{-1}bab$ & $a^{-1}b^{-1}ab$ & $ba^{2}ba^{-1}$ \\
	\hline
	\noalign{\hrule height .04cm}
	\fbox{$W_{2}$} & $1$ & $b^{-1}$ & $ab$ & $b^{2}$ & $a^{-2}$ & $a^{-1}b^{-1}$ & $b^{-1}a$ & $ab^{2}$ & $ab^{-1}a$ & $bab$ \\
	\hline
	$a^{-3}$ & $a^{-1}b^{-1}a$ & $b^{-1}ab$ & $abab$ & $ab^{-1}ab$ & $ba^{2}b$ & $a^{-1}ba^{2}$ & $a^{-1}bab$ & $a^{-1}b^{3}$ & $a^{-1}b^{-1}ab$ & $ba^{2}b^{2}$ \\
	\hline
	\noalign{\hrule height .04cm}
	\fbox{$W_{3}$} & $1$ & $a$ & $ab$ & $b^{2}$ & $a^{-2}$ & $a^{-1}b^{-1}$ & $b^{-1}a$ & $ab^{-1}a$ & $bab$ & $b^{3}$ \\
	\hline
	$a^{-2}b^{-1}$ & $a^{-1}b^{-1}a$ & $b^{-1}ab$ & $abab$ & $ab^{-1}ab$ & $ba^{2}b$ & $a^{-1}ba^{2}$ & $a^{-1}bab$ & $a^{-1}b^{3}$ & $a^{-1}b^{-1}ab$ & $a^{-1}ba^{2}b$ \\
	\hline
	\noalign{\hrule height .04cm}
	\fbox{$W_{4}$} & $1$ & $b^{-1}$ & $ab$ & $b^{2}$ & $ba^{-1}$ & $a^{-2}$ & $a^{-1}b^{-1}$ & $ab^{-1}a$ & $ab^{-1}a^{-1}$ & $ab^{-2}$ \\
	\hline
	$bab$ & $bab^{-1}$ & $a^{-3}$ & $a^{3}b^{-1}$ & $abab$ & $ab^{-1}ab$ & $ab^{-3}$ & $ba^{2}b$ & $a^{-1}bab$ & $a^{-1}b^{-1}ab$ & $a^{2}b^{3}$ \\
	\hline
	\noalign{\hrule height .04cm}
	
\end{tabular}
		\end{center}
	}
\end{table}

There are 80 swap units supported on the ball of \textit{radius $6$ around} $1_P$, The set of  60 swap  units supported on radius 6 but not radius 5 around $1_P$  is partitioned into  12 $T$-orbits of length $4$, and 6 $T$-orbits of length $2$ (equivalently, 3 $S$-orbits of length 4).

All nontrivial units  we displayed  above have a  support of size $21$; the same holds for   Gardam's original one~\cite{G}. However, in radius 6 we found   four swap units with support size 81. Two of them   are displayed  in Table~\ref{table:radius 6} in the appendix, the other two are their  inverses. We also found four swap units of support size 57; the remaining ones have support size 21.

\section{Fibonacci groups}\label{sec_fib}
\noindent We now consider a class of groups many of which are   torsion-free  and not left-orderable;  these groups are potential counterexamples to  the various conjectures. Following Johnson \cite[p.\ 74]{johnson},  for  integers $2\leq r<n$,  the \emph{Fibonacci group} $F(r,n)$  is defined by 
\begin{align}\label{eq_fib}F(r,n)&=\langle\; x_1,\ldots,x_n \mid x_ix_{i+1}\cdots x_{i+r-1}x_{i+r}^{-1}\quad (0\leq i\leq n-1)\;\rangle, 
\end{align}
where subscripts are understood to be modulo $n$ such that all elements $x_j$ lie in  $\{x_1,\ldots,x_n\}$.  These groups and their generalisations   have been studied extensively in the literature; see \cite{kim,thomas} and the references therein for background. We only summarize a few facts.

\begin{enumerate}
	\item   Let $d=\gcd(r+1,n)$. The group $F(r,n)$ is infinite whenever either $d>3$, or $d=3$ and $n$ is even; see \cite{thomas}.  Furthermore,  $F(r,n)$ is infinite whenever $n>5r$;  see \cite[p.\ 76]{johnson}. 
	\item Every finite group is a quotient of some group $F(r,n)$.  
	\item $F(2,6) \cong P$ via $a= x_1 x_2, b= x_1 x_2^2$. 
	\item  $F(2,n)$   has torsion  for odd $n$, $F(2,2m)$ is torsion-free for $m\geq 3$; see  \cite[p.\ 84]{johnson} and \cite[(P4)]{helling}.  
\end{enumerate}
 Fox \cite{fox} proved that $F(2,2m)$ is not right-orderable (which is equivalent to being left-orderable) for $m\geq 2$, and   asked which of these groups satisfy the UPP.

 \subsection{The groups $H_n=F(n-1,n)$ for   $n \ge 4$} 
We write $H_n=F(n-1,n)$  for   $n \ge 4$. By what is said above, such $H_n$ is infinite.   Furthermore, $H_n$ is not left-orderable for $n\geq 4$ by  \cite[Lemma 4]{DPT}. The groups $H_n$ for even $n \ge 4$ are the fundamental groups of 3-manifolds, and torsion free \cite[Proposition 2b.ii]{DPT}. By   recent work~\cite{FisherSanchezPeralta2023} they satisfy the zero divisors conjecture for any domain.

The proof of \cite[Lemma 4]{DPT} mentions that  $x_i^2=x_1\ldots x_n$ for each $i$ in $F(n-1,n)$: For $i=1$ this follows directly from the relator $x_2\ldots x_nx_1^{-1}$; for $i=2$, multiply $x_3\ldots x_nx_1=x_2$ from the left by $x_2$ and use that $x_2\ldots x_n=x_1$, which yields $x_1^2=x_2^2$; the other relations $x_1^2=x_i^2$ follow similarly.  Thus, the subgroup $N$ of $H_n$ generated by $x_1^2$ is central in $H_n$. Tietze transformations \cite[Section 2.4.4]{handbook} can be used to show that
\begin{align}\label{eq_Hn}H_n=\langle x_1,\ldots,x_{n}\mid x_1^2=\ldots=x_n^2=w_n\rangle\quad\text{where}\quad w_n=x_1\ldots x_n.
\end{align}
Let $K_n$ be the free product of the free cyclic groups generated by $x_1,\ldots, x_n$, respectively, with amalgamation $x_1^2=\ldots=x_n^2$, that is,
\begin{align}\label{eq_Kn}K_n=\langle x_1,\ldots,x_{n}\mid x_1^2=\ldots=x_n^2 \rangle. 
\end{align}
Note that $H_n\cong K_n/N$,  where   $N$ is the normal closure of $x_1^{-1}x_2\ldots x_n$ in $K_n$.   It follows from \cite[Theorem 11.68]{rotman} that
\begin{align}\label{eq_fo} g\in K_n\text{ has finite order} \iff g^h\in \langle x_i \rangle \text{ for some $i$ and $h\in K_n$}.
\end{align}
 
Since each group $\langle x_i\rangle$ is torsion-free, this implies  that  $K_n$ is torsion-free as well. We now show that elements in $K_n$ have the following normal form.
\begin{lemma} \label{lem Kn}
There is an algorithm that can rewrite  every element in $K_n$ into a unique normal form  $x_{i_1}\ldots x_{i_k}x_n^z$ where successive $i_u$ and $i_{u+1}$ are distinct, $i_k\ne n$, and $z$ is some integer. 
\end{lemma}

\begin{proof} 
Given that $x_n^2$ is central,  the existence of a normal form follows directly from a  generalisation of \cite[Theorem IV.2.6]{LS} from an amalgamation of $2$ to  $n$ groups. However, to expose the algorithmic content we give a short direct proof. Since each $x_i^2=x_n^2$ is central in $K_n$, each  element in $K_n$ can be transformed into the form   $x_{i_1}\ldots x_{i_k}x_n^z$ as above   by replacing all even powers of $x_j$ by the same  power of $x_n$, and moving all these powers of $x_n$ to the right.  For  the uniqueness, suppose  that $x_{i_1} \ldots x_{i_k}  x_n^z =   x_{j_1} \ldots x_{j_m}  x_n^s$ are two such forms.   If $k=m=0$, then we started with $x_n^z=x_n^s$, which is an equation in the free cyclic group $\langle x_n\rangle$, hence $z=s$. If $k=0$ and $m\ne 0$ (or vice versa), then we have an equation $x_{j_1} \ldots x_{j_m}=  x_n^{z-s}$. Since $m\ne 0$, this forces $s-z=0$ and $x_{j_1} \ldots x_{j_m}=1$, but the latter is not possible by \eqref{eq_fo}. Thus, we can assume that $0<m\leq k$, and rewrite our element equality as 
  \[(\ast)\quad 1=x_{j_m}^{-1}\ldots x_{j_1}^{-1}x_{i_1} \ldots x_{i_k}  x_n^{z-s}=x_{j_m}\ldots x_{j_1}x_{i_1}\ldots x_{i_k} x_n^{z-s-2m}.\]This element has order $1$, and therefore must lie in one of the free factors, say in $\langle x_\ell\rangle$. This is not possible if $j_1\ne i_1$. If $j_1=i_1$, then we can replace $x_{j_1}x_{i_1}=x_n^2$ and move it to the right, and we iterate. So either we can rewrite $(\ast)$ as $1=x_n^{z-s-2m+2m}$, in which case uniqueness follows, or we reach  a word $x_{j_m}\ldots x_{j_g}x_{i_g}\ldots x_{i_k} x_n^{z-s-2m+2(g-1)}=1$ with $x_{j_g}x_{i_g}\notin \{1,x_n^2\}$, or a word $x_{i_{m+1}}\ldots x_{i_k} x_n^{z-s}=1$ with $m+1\leq k$. Both cases are not possible by \eqref{eq_fo}.
\end{proof}

\subsection{Each $H_n$ has a solvable word problem}
Recall that the \textit{word problem} of a finitely presented group $G$  is to decide membership in the set  of  free group words in the generators of the group that  equal the identity in $G$. The word problem is \textit{solvable} if  an algorithm exists to decide this membership. If this set is merely recursively enumerable, one  says that  $G$ is \textit{recursively presented}.

\begin{remark} \label{Malcev} If the word problem for $G$ is solvable, the group is computable in the usual sense of computable algebra: there is a bijection  $\theta \colon G \to \mathbb N$ such that the images under $\theta$ of group products are computable. 
 \end{remark}

\begin{lemma} \label{lem: rec presented} Let $H$ be a recursively presented group and let   $w\in Z(H)$  be a non-trivial central element. If $L= H/\la w \ra$ has a solvable word problem, then so does~$H$.
\end{lemma}
\begin{proof}  Suppose $w$ has order $t\geq 2$, where $t=\infty$ is allowed.  Let $u$ be a word in the generators of $H$. Since the generators of $H$ map onto the generators of $L$, and $L$ has solvable word problem, we can check whether $u$ represents the identity in $L$. If not, we have determined that $u$ is not the identity in $H$. If $u$ does represent the identity in $L$, then we know that $u$ represents a non-trivial element in $\langle w\rangle$, so $u=w^r$ for some  $r\in \ZZ$.  Since $H$ is recursively presented, there is an algorithm that find this $r$. Note that  $r\equiv 0\bmod t$ if and only if $u$ is the identity in $H$. Thus,  we can decide the word problem in~$H$. 
  \end{proof}  

We now show that in $H_n$ we can solve the word problem. While not strictly necessary to compute with $H_n$, having a solution to the word problem does often significantly improve computational capabilities. We continue with the notation of the introduction of this section; recall that $w_n$ is central in $H_n$ since  each $x_i^2=w_n$.

  \begin{theorem}\label{Hn_solv_wp} The group $H_n$ has solvable word problem for $n\geq 4$. \end{theorem}
  
  \begin{proof} Define $L_n = H_n/\la w_n \ra$. In $L_n$, each $x_i^2=1$, and so the relator $w_n=1$ can be written as  $x_1 \ldots x_r =  x_{n}^{-1} \ldots   x_{r+1}^{-1}=  x_{n} \ldots   x_{r+1}$ for any $r\in\{2,\ldots,n-2\}$. Thus, we have a presentation
  \bc $L_n = \la\; x_1,  \dots , x_n \mid x_1^2,\;\ldots,\;x_n^2,\; x_1 \ldots x_r =  x_{n} \ldots   x_{r+1} \;\ra$,  \ec
which shows that $L_n\cong A*_\ZZ B$ is an amalgamated free product where 
 \bc $A= \la\; x_1, \ldots, x_r \mid x_1^2,\ldots,x_r^2 \;\ra \quad\text{and}\quad B= \la\; x_{r+1}, \ldots, x_n \mid x_{r+1}^2,\ldots,x_n^2\; \ra$, \ec
 and $\ZZ$ is embedded into $A $ via $1 \mapsto x_1 \ldots x_r$ and into $B$ via $1 \mapsto   x_{n} \ldots x_{r+1} $. Indeed, since $r$ lies in $\{2,\ldots, n-2\}$, each $A$ and $B$ is a free product of at least two cyclic groups of order $2$, and \eqref{eq_fo} implies that $x_1\ldots x_r$ and $x_{r+1}\ldots x_n$ do not have finite order in $A$ and in $B$, respectively. 
 
 By Lemma \ref{lem: rec presented}, it suffices to show that $L_n$ has solvable word problem. Note that $A,B$ have solvable word problem by an argument similar to the one in Lemma~\ref{lem Kn}. We  now   proceed via the  usual normal form of elements in an amalgam of two groups.  
 Let $U = A\cap B$ and note  \[U=\langle x_1\ldots x_r\rangle = \langle x_{r+1}\ldots x_n\rangle.\] Below we will pick suitable right coset representatives $\mathcal{S}$ of $U$ in $A$,  and $\mathcal{T}$ of $U$ in $B$, both containing~$1$. We first show that each element of $L_n$ can be uniquely written as $u s_1t_1 \ldots s_\ell t_\ell$ where $u\in U$ and each $s_i\in \mathcal{S}$ and $t_i\in \mathcal{T}$, and the $s_i, t_j$ are non-trivial with the possible exceptions $s_1$ and $t_\ell$:
  note that every element in $ A*_\ZZ B$ has the form $a_1b_1\ldots a_\ell b_\ell$ with each $a_i\in A$ and $b_i\in B$; starting from the right, replace $b_\ell = ub'_\ell$ with $u\in U$ and $b'_\ell\in \mathcal{T}$. Recall that $u\in A\cap B$, so next we replace $a_\ell u = va_\ell'$ with $v\in U$ and $a_\ell'\in \mathcal{S}$. An iteration of this process yields the required form. 

In order to   write an algorithm to create and multiply normal forms, we need to discuss the computability of $\mathcal{S}$; the discussion for $\mathcal{T}$ will be similar.  Recall that a  normal form for elements of $A$ is given by words $x_{k_1} \ldots x_{k_m}$ where each $k_i \in \{1, \ldots, r\}$ and $k_i \neq k_{i+1}$. Dropping the $x$'s, we can describe this as a sequence $\sss=[k_1, \ldots, k_m]$ of numbers in $\{1, \ldots, r\}$ as above. Multiplication of two sequences is induced by multiplication by elements in $A$, with the cancellation rules $x_i^2=1$ for each $i$. Thus, given two such sequences $\sigma$ and $\tau$, one can write $\sigma=\sigma'\varrho$ and $\tau=\varrho^{-1}\tau'$ where the subsequence $\varrho$ is as long as possible; in this case, $\sigma\tau=\sigma'\tau'$. (E.g.\ if $\sigma=[4,2,1]$ and $\tau=[1,2,1,3]$, then $\varrho=[2,1]$ and  $\sigma\tau=[4,1,3]$.) In the following we  use the notion of computable sets, implicitly assuming an encoding of group elements by natural numbers according to Remark~\ref{Malcev}. The subgroup $U = \la x_1 \ldots x_r \ra$ of $A$  is   computable, because given a normal form for $A$ one can check whether it  represents a power of $x_1 \ldots x_r$. As   a right transversal $\mathcal{S}$ of $U$ in $A$, we now pick those elements $\sss$ that are the length-lexicographically least in their right cosets of $\la x_1 \ldots x_r\ra  $. The latter is decidable because $\mathcal{S}$ consists of  all the sequences $\sss$ that do not start with $[1, 2,\ldots, s]$ or with $[r,r-1,\ldots,r-t]$ for some $s> r/2$ or $t< r/2$: to see this, observe that  \[\langle x_1\ldots x_r\rangle x_1\ldots x_s = \langle x_1\ldots x_r\rangle x_r x_{r-1}\ldots x_{s+1},\] and $[r,r-1,\ldots,s+1]$ is lexicographically smaller than $[1,2,\ldots,s]$; analogously for the representative $x_rx_{r-1}\ldots x_{r-t}$.  Since $r$ is fixed, $\mathcal{S}$ is computable. A similar argument for $\mathcal{T}$ (with variables $y_i = x_{n-i+1}$ for $1\le i \le n-r$) shows that $\mathcal{T}$ is computable. Thus $L_n$ has solvable word problem, and  so does  $H_n$.
\end{proof}

\section{the structure of $H_4$}\label{sec_nut}
\noindent We use the common notation ${\rm pc}\la X\mid R\ra$ for a polycyclic group presentation $\la X\mid R\rangle$ where all trivial commutator relations $g_i^{g_j}=g_i$ are omitted (where $g_i,g_j\in X$), c.f.\ \cite[Section 8.1]{handbook}.

\begin{proposition}\label{prop_pc}The group $H_4$ is isomorphic to the (torsion-free) polycyclic group
\[ 
{\rm pc}\la r,a,b,z  \mid r^2=z,\;\; a^r = a^{-1},\;\; b^r = b^{-1},\;\; b^a =bz^2  \ra.
\]
\end{proposition}

\begin{proof} 
Let $H$ be the group given by the polycyclic presentation in the proposition, with generators $r,a,b,z$.    Define a map $\psi\colon \{r,a,b,z\}\to \{\tilde r,\tilde a, \tilde b, \tilde z\}\subseteq H_4$, where
  \[
  \tilde r=x_1,\quad \tilde a= x_1x_4^{-1}, \quad\tilde b=x_1x_2x_4^{-2}, \quad \tilde z=x_1x_2x_3x_4.
  \]
  One can show that $\{\tilde r,\tilde a,\tilde b,\tilde z\}$ generates $H_4$. Recall that each   $x_i^2=x_1x_2x_3x_4$, so $\tilde z$ is central and $\tilde r^2=\tilde z$. We show that the relations of $H$ hold in $H_4$ via $\psi$; then von Dyck's Theorem \cite[Theorem 2.53]{handbook} proves that $\psi$ extends to a unique group epimorphism $H\to H_4$. For this it remains to consider the images of the relations involving $a^{r}$, $b^r$, and $b^a$. Note that $\tilde a^{\tilde r}=x_4^{-1}x_1=x_4x_1^{-1}=\tilde a^{-1}$ where the middle equation holds since   $x_4^{-2}x_1^{2}=1$. Similarly, we have $\tilde b^{\tilde r} = x_2x_4^{-2}x_1 = x_2^{-1} x_1^{-1}x_4^2 = \tilde b^{-1}$, where the middle equation holds since $x_1x_2^2x_4^{-2}x_1=x_1^2=x_4^2$. Lastly, since $x_1x_2x_3=x_4$, we have
  \[\tilde b^{\tilde a}=x_4 x_2 x_4^{-2}x_1x_4^{-1}=x_4^{-1}x_2x_1x_4^{-1} \qquad \textrm{and} \qquad \tilde b\tilde z^2=x_1x_2x_1x_2x_3x_4=x_1x_2x_3^2=x_4x_3;\]
 now $\tilde b^{\tilde a}=\tilde b\tilde z^2$ follows from $x_1=x_2x_3x_4$ and $x_3=x_2^{-1}x_1x_4^{-1}=x_2^{-2}x_2x_1x_4^{-1}=x_4^{-2}x_2x_1x_4^{-1}$. Now von Dyck implies that $\psi$ induces an epimorphism $H\to H_4$. To prove isomorphism, consider the map $\varphi\colon \{\tilde r,\tilde a,\tilde b,\tilde z\}\to \{r,a,b,z\}$. To apply von Dyck's Theorem to the generators $x_1,\ldots,x_4$, note that
  \[x_1=\tilde r,\quad x_2=\tilde r^{-1} \tilde b (\tilde a^{-1}\tilde r)^2 ,\quad  x_3= (\tilde a^{-1}\tilde r)^{-2}\tilde b^{-1}\tilde z \tilde r^{-1}\tilde a,\quad x_4=\tilde a^{-1}\tilde r,\]so $\varphi$ translates to the map \[\varphi'\colon \{x_1,x_2,x_3,x_4\}\to \{r, r^{-1} b ( a^{-1} r)^2, (a^{-1} r)^{-2}b^{-1} z r^{-1} a, a^{-1}r \}.\] A straightforward calculation in the polycyclic group $H$ shows that the image of $\varphi'$ satisfies the relations of $H_4$. Since $\varphi$ and $\psi$ are mutually inverse, $H\cong H_4$.  
\end{proof}

 \section{Normal forms for  $\mathbb{F}_2[H_4]$}\label{sec_nut2}
\label{sec_norm} 
\noindent Using SAT solvers, we  showed that there is  no  non-trivial unit so that both the unit and its inverse or supported on a ball of radius 4 in the Cayley graph given by the generators $a,b,r$. For larger radius our query did not return  an answer. Here we develop    normal forms for the group ring $\mathbb{F}_2[H_4]$, which could also be used to search for nontrivial units.

Let $K$ be a field. We   consider normal forms for $K[H_4]$, and for this we first define a subgroup $S<H_4$ of index $2$. Recall that $[r,a,b,z]$ is a polycyclic generating set for $H_4$, so $[a,b,z]$ is a polycyclic generating set for the subgroup $S\leq H_4$ generated by $\{a,b,z\}$; in particular, \[S\cong {\rm pc}\langle a,b,z \mid b^a = bz^2\rangle\] has index $2$ in $H_4$ and is torsion-free. The normal forms of $S$ are exactly $a^ub^vz^w$ with $u,v,w\in\Z$. Using the presentation of $S$, a simple calculation shows that
\begin{equation} \label{eqn:rst}(a^ub^vz^w)(a^{u'}b^{v'}z^{w'})= a^{u+u'}b^{v+v'}z^{w+w'+2vu'}.
\end{equation}
Note that $r^2=z$ is central, so conjugation with $r$ is an automorphism of order $2$ of  $S$. Specifically, we have
\begin{equation} \label{eqn:drst}(a^u b^v z^w)^r = a^{-u} b^{-v} z^{w}.
\end{equation}
Since $H_4$ is the disjoint union of $S$ and $rS$, we have the following.

\begin{lemma} Every element $u \in  K[H_4]$ can be written uniquely as $u= \alpha + r\beta$ with $\alpha,\beta \in K[S]$.
\end{lemma}
For $\beta\in K[S]$ write $\beta^r=r^{-1}\beta r$; this can be evaluated via \eqref{eqn:drst}. If $\alpha+r\beta$ and $\alpha'+r\beta'$ lie in $K[H_4]$, 
\begin{align}\label{eq_nfs}(\alpha+r\beta)(\alpha'+r\beta')&=\alpha\alpha' + z \beta^r\beta' +r(\beta\alpha' + \alpha^r\beta');
\end{align}
recall that $r^2=z$, so $z\beta^r\beta'=r^2r^{-1}\beta r \beta'=r\beta r \beta'$.

Now consider $K=\mathbb{F}_2$, the binary field. The normal forms exhibited above can be used for the  search for non-trivial units: For $N\in\mathbb N$ denote by $B_N$ the set of all elements  $a^ub^vz^w\in S$ with $|u|+ |v| + |w| \le N$; by abuse of notation, we call this a ball of radius $N$ for $S$.  Now  think of making  a guess for  $\alpha,\beta,\alpha',\beta'\in \mathbb{F}_2[S]$ by picking four finite subsets of $B_N$; that is, assign truth values to Boolean variables,  one per ball element. (It may be useful to only consider subsets of small sizes; in Gardam's work \cite{G}, the support sizes are $21$ out of a ball of $144$ elements.) By \eqref{eq_nfs}, the elements $u=\alpha+r\beta$ and $v=\alpha'+r\beta'$ are units if and only if 
\[\alpha\alpha' + z \beta^r\beta'=1\quad\text{and}\quad \beta\alpha' + \alpha^r\beta'=0.\]
Since $\alpha,\alpha',\beta,\beta'$ are indeterminates, we can represent these equations by a Boolean formula. Since $S$ is left-orderable, $\mathbb{F}_2[S]$ satisfies the unit conjecture. So we can impose that one  of the variables in $\beta$ and $\beta'$ is set to~$1$. A SAT solver might be used    to  see if there are  $\alpha,\alpha',\beta,\beta'$ constituting non-trivial units. 

 \section{The unique product property}

\subsection{Refuting the   UPP: a computational approach}\label{sec_comp}

\noindent Let $G$ be a group. If  $K=\mathbb{F}_2$ is the field with $2$ elements, there is a one-to-one correspondence between subsets $S\subseteq G$ and elements $g_S = \sum_{g\in S} g$ in the group ring $K[G]$. If $G$ does not  satisfy the UPP,  there exist finite non-empty subsets $A,B\subseteq G$ such that every element in the multiset $AB$ occurs with multiplicity at least $2$. Thus, if $C$ is the set of all elements in $AB$, then in the product 
\[g_A g_B = \sum_{c\in C} (\sum_{a\in A, b\in B, \atop ab=c} 1) c\]
each  coefficient $\sum_{a\in A, b\in B, ab=c} 1$ is greater than $1$.

We now describe how to use a SAT solver to potentially find two such sets $A$ and $B$ within a prescribed finite super-set $S\subseteq G$. Let $S=\{s_1,\ldots,s_n\}$ and identify $A$ and $B$ with $a=g_A=\sum_{s\in S} a_s s$ and $b=g_B=\sum_{s\in S}b_s s$, respectively, where each $a_s$ and $b_s$ is an indeterminate in $\mathbb{F}_2$, considered as a Boolean variable (with $1={\tt true}$ and $0={\tt false}$). We want to find assignments $a_s,b_s\in\{0,1\}$ for each $s\in S$ such that
\begin{align}\label{eq_SAT_1}\Big(\bigvee_{s\in S} a_s\Big)\wedge \Big(\bigvee_{s\in S} b_s\Big)\wedge\bigwedge_{u,v\in S} \Big( (a_u\wedge b_v) \rightarrow \bigvee_{u',v'\in S\atop {u'\ne v \atop uv=u'v'}} a_{u'}\wedge b_{v'}\Big)
\end{align} 
is true. If we find such an assignment, then the first two parts of that formula say that the  corresponding elements $a$ and $b$ are non-zero in $\mathbb{F}_2[G]$, whereas the last part says that whenever $a_u=b_v=1$, then the coefficient of $uv$ in the sum describing $ab$ has at least two non-zero terms $a_ub_v$ and $a_{u'}b_{v'}$, that is, $uv=u'v'$ with $u\ne u'$. To bring \eqref{eq_SAT_1}  into conjunctive normal form, we introduce auxiliary Boolean variables $c_{u,v}$ defined by the property $c_{u,v} \leftrightarrow (a_u\wedge b_v)$; this turns \eqref{eq_SAT_1} into the following formula with indeterminates $a_u, b_v, c_{u,v}$ for $u,v\in S$:
\begin{equation}\label{eq_SAT}
	\begin{array}{ll}&\hspace*{1.8ex}\bigwedge_{u,v\in S} \Big((\neg c_{u,v}\vee a_u) \wedge (\neg c_{u,v} \vee a_v)\wedge (\neg a_u\vee \neg  a_v\vee c_{u,v})\Big)\\[2ex]\wedge&\Big(\bigvee_{s\in S} a_s\Big)\wedge \Big(\bigvee_{s\in S} b_s\Big)\wedge \bigwedge_{u,v\in S} \Big( (\neg c_{u,v}) \vee  \bigvee_{u',v'\in S\atop {u'\ne v\atop uv=u'v'}} c_{u',v'}\Big).
	\end{array}
\end{equation}

For a given group $G$ and given set $S$, we compute the multiplication table with rows and columns labelled by the elements of $S$, and use this table to write down the clauses of the formula \eqref{eq_SAT}. If a SAT solver determines that this formula can be satisfied, then the corresponding solution determines two sets $A$ and $B$ that demonstrate that $G$ does not satisfy the UPP.

\subsection{The UPP fails for $H_4$}

\begin{proposition}\label{prop_H4upp} The group $H_4$ does not satisfy the unique product property.
\end{proposition}
\begin{proof}
  Let $H=H_4$ and consider the ball of radius $3$ in the Cayley graph of $H$ with respect to the generators $\{x_1,\ldots,x_4\}$. We used the SAT solver \textcode{Kissat} to show that the formula \eqref{eq_SAT} is satisfiable; our solutions translates to the sets $A$ and $B$ shown in Figure \ref{fig_G4}, which also provides GAP code that verifies the claim.
\end{proof}

\begin{figure}[h] 
\begin{verbatim} 
  F  := FreeGroup(["x1","x2","x3","x4"]);; 
  AssignGeneratorVariables(F);;
  R  := [x2*x3*x4/x1, x3*x4*x1/x2, x4*x1*x2/x3, x1*x2*x3/x4];;
  H4 := F/R;;
  AssignGeneratorVariables(H4);;
  A := [ x1^0, x1, x4, x1^-1, x3^-1, x1^2, x1*x3, x1*x2^-1, x1*x3^-1,
         x1*x4^-1, x2*x1, x2*x4^-1, x3*x1^-1, x4*x3, x4*x1^-1,  x1^3,
         x4*x2^-1, x1^-2, x1^-1*x2^-1, x3^-1*x4^-1, x3*x4^-1*x1^-1,
         x1*x3*x2^-1, x1*x4*x1^-1, x1*x4^-1*x2^-1, x2*x1*x2, x1^2*x2,
         x2*x4^-1*x2^-1, x3*x1*x2^-1,  x3^-1*x4^-1*x2^-1 ];;
  B := [ x1^0, x1, x3, x1^-1, x3^-1, x4^-1, x1*x3^-1, x1*x4^-1, 
         x2*x4, x2*x1^-1, x2*x3^-1, x3*x2^-1, x3*x4^-1, x4*x3, 
         x1^-1*x2^-1, x2^-1*x3^-1, x2^-1*x4^-1, x1^2*x3, x1^2*x4,
         x2*x1*x2, x2*x1*x3^-1, x2*x4*x1^-1, x2*x1^-1*x2^-1,
         x2*x3^-1*x4^-1, x2*x4^-1*x2^-1, x2*x1, x4*x2^-1 ];;
  ForAll(Collected(Flat(List(A,a->List(B,b->a*b)))),u->u[2]>1);
  # true
\end{verbatim}
\caption{GAP code to verify that $H_4$ does not satisfy the UPP.}\label{fig_G4}
\end{figure}

 We note that the failure of the UPP in Proposition \ref{prop_H4upp} is not a property  inherited from the group $P$ in \eqref{eq_P}. Recall that the Hirsch length of a polycyclic group is the number of infinite factors in a polycyclic series for the group; this number is independent of the chosen polycyclic series.

\begin{lemma} 
The group $P$ from \eqref{eq_P} is not isomorphic to a quotient of a subgroup of $H_4$. 
\end{lemma}
\begin{proof}It is known that $P$ contains a subgroup isomorphic to $\Z^3$, see \cite[p.\ 447]{C}. Thus, if $P$ is isomorphic to a quotient of a subgroup of $H_4$, then there exist  $U\leq H_4$ and $V\unlhd U$ with $\Z^3\cong U/V$. The group  $H_4$ has Hirsch length $3$, and it follows from \cite[p.\ 16]{segal} that the Hirsch length of $U$ is at most $3$, with equality if and only if $U$ has finite index in $H_4$. Since $U/V\cong \Z^3$, this forces $[H_4:U]<\infty$ and $U$ has Hirsch length~$3$. Since $V$ is normal in $U$, it also follows from  \cite[p.\ 16]{segal} that the Hirsch length of $U$ is the sum of the Hirsch lengths of $U/V$ and $V$. This forces that $V$ has Hirsch length $0$, so $V$ is finite. Since $H_4$ is torsion-free, $V=1$, and therefore $U\cong \Z^3$ is a finite index subgroup of $H_4$. We show that this is not possible. 

  For a contradiction, suppose that $x=r^wa^sb^tz^u$ and $x'=r^{w'}a^{s'}b^{t'}z^{u'}$ and $x''=r^{w''}a^{s''}b^{t''}z^{u''}$ are generators of such a subgroup $U\cong \Z^3$ in normal form. Recall that $w,w',w''\in\{0,1\}$. We first show that $\langle x,x',x''\rangle\cong \Z^3$ forces $w=w'=w''=0$. Suppose not all $w,w',w''$ are $0$, say $w=1$. We can now arrange that $w'=w''=0$ by replacing $x'$ and $x''$ by a product with  $x^{-1}$, if necessary. Now the relations of $H_4$ show that the exponents of $a$ and $b$ in $(x')^x$ are ${-s'}$ and ${-t'}$, respectively, but since $x$ and $x'$ commute, this forces $s'=t'=0$. This implies that if $w=1$, then $x'\in \langle z\rangle$, and the same argument shows that $x''\in\langle z\rangle$. But then $\langle x,x',x''\rangle \cong \Z^3$ is not possible, a contradiction. Thus, we can assume that $w=w'=w''=0$, so $x,x',x''\in S=\langle a,b,z\rangle$.

  Now we repeat with a similar argument. By assumption, $\Z^3\cong \langle x,x',x''\rangle\leq S$. Since $\langle b,z\rangle \cong \Z^2$, we can assume that $x\notin \langle b,z\rangle$, that is, $s\ne 0$. As before, we can  arrange that $s'=s''=0$ by replacing $x'$ and $x''$ by suitable powers and a product with a power of $x$. That is, we assume $\langle x,x',x''\rangle\cong \Z^3$ with $x',x''\in\langle b,z\rangle$ and $x\in S\setminus\langle b,z\rangle$.  Since $\langle b,z\rangle$ is abelian, we now observe $x'=(x')^x=(x')^{a^sb^tz^u}=(b^{t'}z^{u'})^{a^s}=b^{t'}z^{2st'+u'}$, and $s\ne 0$ forces $t'=0$. But then $x'\in\langle z\rangle$, and a similar argument shows that $x''\in\langle z\rangle$. But then  $\langle x,x',x''\rangle\cong \Z^3$ is not possible, a final contradiction. 
\end{proof}

   \setlength{\fboxsep}{1pt}
 \setlength{\fboxrule}{0.8pt}
 \begin{table}[ht]
 	\caption{Two swap units in a ball of radius 6 with support size 81 (see \ref{ss:rad5and6})} 
 	\label{table:radius 6}
 	{\fontsize{7}{8.5}\selectfont
 		\begin{center}
 			\setlength{\tabcolsep}{3pt}
 	
 			\begin{tabular}{|c|c|c|c|c|c|c|c|}
 			\hline
 			\fbox{$S_1$} & $1$ & $a$ & $b$ & $a^{-1}$ & $b^{-1}$ & $a^{3}$ & $a^{2}b$ \\
 			\hline
 			$a^{2}b^{-1}$ & $ab^{2}$ & $aba^{-1}$ & $ab^{-2}$ & $ba^{2}$ & $bab^{-1}$ & $b^{3}$ & $b^{2}a^{-1}$ \\
 			\hline
 			$a^{-1}b^{2}$ & $a^{-3}$ & $a^{-2}b^{-1}$ & $a^{-1}b^{-1}a$ & $b^{-1}a^{-1}b$ & $b^{-3}$ & $a^{4}$ & $a^{2}ba^{-1}$ \\
 			\hline
 			$a^{2}b^{-1}a^{-1}$ & $aba^{-1}b^{-1}$ & $ab^{-1}a^{-1}b$ & $ba^{3}$ & $bab^{-1}a^{-1}$ & $bab^{-2}$ & $b^{4}$ & $b^{3}a^{-1}$ \\
 			\hline
 			$ba^{-1}b^{-1}a$ & $a^{-1}bab^{-1}$ & $a^{-4}$ & $a^{-2}b^{-1}a$ & $a^{-1}b^{-1}ab$ & $b^{-1}ab^{2}$ & $b^{-1}aba^{-1}$ & $b^{-1}a^{-1}ba$ \\
 			\hline
 			$b^{-1}a^{-1}b^{2}$ & $b^{-4}$ & $a^{4}b^{-1}$ & $a^{3}b^{2}$ & $a^{3}ba^{-1}$ & $a^{2}b^{-3}$ & $ababa$ & $ababa^{-1}$ \\
 			\hline
 			$abab^{-1}a^{-1}$ & $ab^{4}$ & $aba^{-1}ba$ & $aba^{-1}b^{-1}a$ & $ab^{-1}aba$ & $ba^{4}$ & $ba^{2}b^{2}$ & $ba^{2}ba^{-1}$ \\
 			\hline
 			$babab$ & $babab^{-1}$ & $baba^{-1}b^{-1}$ & $bab^{-1}ab$ & $bab^{-1}a^{-1}b$ & $bab^{-3}$ & $b^{4}a^{-1}$ & $ba^{-1}bab$ \\
 			\hline
 			$a^{-1}baba^{-1}$ & $a^{-3}b^{-1}a$ & $a^{-1}b^{-1}aba^{-1}$ & $b^{-1}abab^{-1}$ & $b^{-1}a^{-1}bab^{-1}$ & $b^{-1}a^{-1}b^{3}$ & $a^{2}ba^{-1}ba$ & $ba^{3}ba^{-1}$ \\
 			\hline
 			$baba^{-1}ba$ & $baba^{-1}b^{-1}a$ & $bab^{-1}aba$ & $bab^{-1}ab^{2}$ & $bab^{-1}aba^{-1}$ & $ba^{-1}baba^{-1}$ & $ba^{-1}bab^{-1}a^{-1}$ & $b^{-1}abab^{-1}a$ \\
 			\hline
 			$b^{-1}abab^{-1}a^{-1}$ & $b^{-1}abab^{-2}$ \\
 			\hline
 			\noalign{\hrule height .04cm}
 			\fbox{$S_2$} & $1$ & $a$ & $b$ & $a^{-1}$ & $b^{-1}$ & $a^{3}$ & $a^{2}b$ \\
 			\hline
 			$a^{2}b^{-1}$ & $ab^{2}$ & $ab^{-1}a^{-1}$ & $ab^{-2}$ & $ba^{2}$ & $b^{3}$ & $b^{2}a^{-1}$ & $ba^{-1}b^{-1}$ \\
 			\hline
 			$a^{-1}ba$ & $a^{-1}b^{2}$ & $a^{-3}$ & $a^{-2}b^{-1}$ & $b^{-1}ab$ & $b^{-3}$ & $a^{4}$ & $a^{3}b$ \\
 			\hline
 			$a^{3}b^{-1}$ & $ab^{3}$ & $aba^{-1}b^{-1}$ & $ab^{-1}a^{-1}b$ & $ab^{-3}$ & $bab^{-1}a^{-1}$ & $b^{4}$ & $b^{2}a^{-1}b^{-1}$ \\
 			\hline
 			$ba^{-1}b^{-1}a$ & $a^{-1}ba^{2}$ & $a^{-1}bab^{-1}$ & $a^{-1}b^{3}$ & $a^{-4}$ & $a^{-3}b^{-1}$ & $a^{-1}b^{-1}ab$ & $b^{-1}aba^{-1}$ \\
 			\hline
 			$b^{-1}a^{-1}ba$ & $b^{-4}$ & $a^{4}b$ & $a^{3}b^{-1}a^{-1}$ & $a^{3}b^{-2}$ & $a^{2}b^{3}$ & $ababa$ & $ababa^{-1}$ \\
 			\hline
 			$abab^{-1}a$ & $aba^{-1}ba$ & $ab^{-1}aba^{-1}$ & $ab^{-1}a^{-1}ba$ & $ab^{-4}$ & $babab$ & $babab^{-1}$ & $baba^{-1}b$ \\
 			\hline
 			$bab^{-1}ab$ & $b^{3}a^{-1}b^{-1}$ & $ba^{-1}bab^{-1}$ & $ba^{-1}b^{-1}ab$ & $a^{-1}ba^{3}$ & $a^{-1}ba^{2}b$ & $a^{-1}baba^{-1}$ & $a^{-1}bab^{-1}a^{-1}$ \\
 			\hline
 			$a^{-1}b^{4}$ & $a^{-1}b^{2}a^{-1}b^{-1}$ & $a^{-4}b^{-1}$ & $b^{-1}abab^{-1}$ & $b^{-1}ab^{3}$ & $b^{-1}aba^{-1}b^{-1}$ & $a^{3}ba^{-1}b$ & $abab^{-1}ab$ \\
 			\hline
 			$abab^{-1}a^{-1}b$ & $abab^{-3}$ & $aba^{-1}bab$ & $aba^{-1}bab^{-1}$ & $ab^{-1}abab^{-1}$ & $ab^{-1}ab^{3}$ & $ab^{-1}aba^{-1}b^{-1}$ & $a^{-1}ba^{3}b$ \\
 			\hline
 			$a^{-1}baba^{-1}b$ & $a^{-1}baba^{-1}b^{-1}$ \\
 			\hline
 			\end{tabular}
 		\end{center}  
 	}
 \end{table}

 \setlength{\fboxsep}{1pt}
\setlength{\fboxrule}{0.8pt}
\begin{table}[ht]
	\caption{All nontrivial units in a ball of radius 4,  where $U_iV_i=1$ (see \ref{ss:rad4}).} 
	\label{table:radius 4}
	{\fontsize{7}{8.5}\selectfont
		\begin{center}
			\setlength{\tabcolsep}{3pt}
	
			\begin{tabular}{|c|c|c|c|c|c|c|c|c|c|c|}
			\hline
			\fbox{$U_{1}$} & $a^{2}$ & $ab$ & $ab^{-1}$ & $ba$ & $b^{2}$ & $ba^{-1}$ & $a^{-1}b$ & $a^{-2}$ & $a^{-1}b^{-1}$ & $b^{-1}a$ \\
			\hline
			$b^{-1}a^{-1}$ & $b^{-2}$ & $a^{2}b$ & $aba$ & $ab^{-1}a$ & $ab^{-2}$ & $bab$ & $ba^{-1}b$ & $a^{-1}b^{2}$ & $a^{-2}b^{-1}$ & $baba$ \\
			\noalign{\hrule height .03cm}
			\fbox{$V_{1}$} & $a^{2}$ & $ab$ & $ab^{-1}$ & $ba$ & $b^{2}$ & $ba^{-1}$ & $a^{-1}b$ & $a^{-2}$ & $a^{-1}b^{-1}$ & $b^{-1}a$ \\
			\hline
			$b^{-1}a^{-1}$ & $b^{-2}$ & $a^{2}b$ & $aba$ & $ab^{-1}a$ & $ab^{-2}$ & $bab$ & $ba^{-1}b$ & $a^{-1}b^{2}$ & $a^{-2}b^{-1}$ & $abab$ \\
			\noalign{\hrule height .04cm}
			\fbox{$U_{2}$} & $a^{2}$ & $ab$ & $ab^{-1}$ & $ba$ & $b^{2}$ & $ba^{-1}$ & $a^{-1}b$ & $a^{-2}$ & $a^{-1}b^{-1}$ & $b^{-1}a$ \\
			\hline
			$b^{-1}a^{-1}$ & $b^{-2}$ & $a^{2}b^{-1}$ & $aba$ & $ab^{2}$ & $ab^{-1}a$ & $ba^{2}$ & $bab$ & $b^{2}a^{-1}$ & $ba^{-1}b$ & $baba$ \\
			\noalign{\hrule height .03cm}
			\fbox{$V_{2}$ }& $a^{2}$ & $ab$ & $ab^{-1}$ & $ba$ & $b^{2}$ & $ba^{-1}$ & $a^{-1}b$ & $a^{-2}$ & $a^{-1}b^{-1}$ & $b^{-1}a$ \\
			\hline
			$b^{-1}a^{-1}$ & $b^{-2}$ & $a^{2}b^{-1}$ & $aba$ & $ab^{2}$ & $ab^{-1}a$ & $ba^{2}$ & $bab$ & $b^{2}a^{-1}$ & $ba^{-1}b$ & $abab$ \\
			\noalign{\hrule height .04cm}
			\fbox{$U_{3}$} & $b$ & $a^{-1}$ & $b^{-1}$ & $ba^{-1}$ & $a^{-1}b$ & $a^{-1}b^{-1}$ & $b^{-2}$ & $a^{3}$ & $aba$ & $ab^{2}$ \\
			\hline
			$ab^{-1}a$ & $ab^{-2}$ & $ba^{2}$ & $ba^{-1}b$ & $a^{-1}ba$ & $a^{-2}b^{-1}$ & $a^{-1}b^{-1}a$ & $a^{2}b^{2}$ & $baba$ & $ba^{-1}ba$ & $a^{-2}b^{-1}a$ \\
			\noalign{\hrule height .03cm}
			\fbox{$V_{3}$} & $a$ & $b$ & $b^{-1}$ & $ab$ & $ba$ & $b^{-1}a$ & $b^{-2}$ & $aba$ & $ab^{-1}a$ & $ba^{2}$ \\
			\hline
			$bab$ & $b^{2}a^{-1}$ & $a^{-1}ba$ & $a^{-1}b^{2}$ & $a^{-3}$ & $a^{-2}b^{-1}$ & $a^{-1}b^{-1}a$ & $abab$ & $ba^{2}b$ & $a^{-1}bab$ & $a^{-3}b^{-1}$ \\
			\noalign{\hrule height .04cm}
			\fbox{$U_4$} & $b$ & $a^{-1}$ & $b^{-1}$ & $b^{2}$ & $a^{-1}b$ & $a^{-1}b^{-1}$ & $b^{-1}a^{-1}$ & $a^{3}$ & $aba$ & $ab^{2}$ \\
			\hline
			$ab^{-1}a$ & $ab^{-2}$ & $ba^{2}$ & $ba^{-1}b$ & $a^{-1}ba$ & $a^{-2}b^{-1}$ & $a^{-1}b^{-1}a$ & $a^{2}b^{-2}$ & $ba^{3}$ & $baba$ & $ba^{-1}ba$ \\
			\noalign{\hrule height .03cm}
			\fbox{$V_{4}$ }& $a$ & $b$ & $b^{-1}$ & $ab^{-1}$ & $ba$ & $b^{2}$ & $b^{-1}a$ & $aba$ & $ab^{-1}a$ & $ba^{2}$ \\
			\hline
			$bab$ & $b^{2}a^{-1}$ & $a^{-1}ba$ & $a^{-1}b^{2}$ & $a^{-3}$ & $a^{-2}b^{-1}$ & $a^{-1}b^{-1}a$ & $abab$ & $a^{-1}ba^{2}$ & $a^{-1}bab$ & $a^{-1}b^{2}a^{-1}$ \\
			\noalign{\hrule height .04cm}
			\fbox{$U_{5}$} & $b$ & $a^{-1}$ & $b^{-1}$ & $ba^{-1}$ & $a^{-1}b$ & $b^{-1}a^{-1}$ & $b^{-2}$ & $a^{3}$ & $a^{2}b$ & $a^{2}b^{-1}$ \\
			\hline
			$aba$ & $ab^{2}$ & $aba^{-1}$ & $ab^{-1}a$ & $ab^{-1}a^{-1}$ & $ab^{-2}$ & $ba^{-1}b$ & $a^{3}b^{-1}$ & $a^{2}b^{2}$ & $abab$ & $aba^{-1}b$ \\
			\noalign{\hrule height .03cm}
			\fbox{$V_{5}$} & $a$ & $b$ & $b^{-1}$ & $ab$ & $ab^{-1}$ & $ba$ & $b^{-2}$ & $a^{2}b$ & $a^{2}b^{-1}$ & $aba$ \\
			\hline
			$aba^{-1}$ & $ab^{-1}a$ & $ab^{-1}a^{-1}$ & $bab$ & $b^{2}a^{-1}$ & $a^{-1}b^{2}$ & $a^{-3}$ & $a^{2}b^{-1}a^{-1}$ & $ba^{2}b$ & $baba$ & $baba^{-1}$ \\
			\noalign{\hrule height .04cm}
			\fbox{$U_{6}$} & $b$ & $a^{-1}$ & $b^{-1}$ & $ab^{-1}$ & $ba$ & $b^{-1}a$ & $b^{-2}$ & $a^{3}$ & $aba$ & $ab^{2}$ \\
			\hline
			$ab^{-1}a$ & $ab^{-2}$ & $ba^{2}$ & $ba^{-1}b$ & $a^{-1}ba$ & $a^{-2}b^{-1}$ & $a^{-1}b^{-1}a$ & $a^{2}b^{2}$ & $abab$ & $aba^{-1}b$ & $a^{-1}ba^{2}$ \\
			\noalign{\hrule height .03cm}
			\fbox{$V_{6}$} & $a$ & $b$ & $b^{-1}$ & $a^{-1}b$ & $a^{-1}b^{-1}$ & $b^{-1}a^{-1}$ & $b^{-2}$ & $aba$ & $ab^{-1}a$ & $ba^{2}$ \\
			\hline
			$bab$ & $b^{2}a^{-1}$ & $a^{-1}ba$ & $a^{-1}b^{2}$ & $a^{-3}$ & $a^{-2}b^{-1}$ & $a^{-1}b^{-1}a$ & $ba^{3}$ & $ba^{2}b$ & $baba$ & $baba^{-1}$ \\
			\noalign{\hrule height .04cm}
			\fbox{$U_{7}$} & $b$ & $a^{-1}$ & $b^{-1}$ & $ab$ & $ba$ & $b^{2}$ & $b^{-1}a$ & $a^{3}$ & $aba$ & $ab^{2}$ \\
			\hline
			$ab^{-1}a$ & $ab^{-2}$ & $ba^{2}$ & $ba^{-1}b$ & $a^{-1}ba$ & $a^{-2}b^{-1}$ & $a^{-1}b^{-1}a$ & $a^{2}b^{-2}$ & $abab$ & $aba^{-1}b$ & $a^{-3}b^{-1}$ \\
			\noalign{\hrule height .03cm}
			\fbox{$V_{7}$} & $a$ & $b$ & $b^{-1}$ & $b^{2}$ & $ba^{-1}$ & $a^{-1}b$ & $a^{-1}b^{-1}$ & $aba$ & $ab^{-1}a$ & $ba^{2}$ \\
			\hline
			$bab$ & $b^{2}a^{-1}$ & $a^{-1}ba$ & $a^{-1}b^{2}$ & $a^{-3}$ & $a^{-2}b^{-1}$ & $a^{-1}b^{-1}a$ & $baba$ & $baba^{-1}$ & $a^{-1}b^{2}a^{-1}$ & $a^{-2}b^{-1}a$ \\
			\noalign{\hrule height .04cm}
			\fbox{$U_{8}$} & $a$ & $a^{-1}$ & $b^{-1}$ & $a^{2}$ & $a^{-1}b^{-1}$ & $b^{-1}a$ & $b^{-1}a^{-1}$ & $a^{2}b$ & $ab^{2}$ & $ab^{-1}a$ \\
			\hline
			$ba^{2}$ & $bab$ & $b^{3}$ & $ba^{-1}b$ & $a^{-1}b^{2}$ & $b^{-1}ab$ & $b^{-1}a^{-1}b$ & $abab$ & $ab^{3}$ & $ab^{-1}ab$ & $ba^{2}b$ \\
			\noalign{\hrule height .03cm}
			\fbox{$V_{8}$} & $a$ & $b$ & $a^{-1}$ & $a^{2}$ & $ab$ & $ba^{-1}$ & $a^{-1}b$ & $a^{2}b^{-1}$ & $aba$ & $ab^{2}$ \\
			\hline
			$bab$ & $ba^{-1}b$ & $a^{-1}b^{2}$ & $a^{-2}b^{-1}$ & $b^{-1}ab$ & $b^{-1}a^{-1}b$ & $b^{-3}$ & $baba$ & $a^{-1}b^{2}a^{-1}$ & $b^{-1}aba$ & $b^{-1}ab^{2}$ \\
			\noalign{\hrule height .04cm}
			\fbox{$U_{9}$} & $a$ & $a^{-1}$ & $b^{-1}$ & $ba$ & $ba^{-1}$ & $a^{-1}b$ & $a^{-2}$ & $a^{2}b$ & $ab^{-1}a$ & $ab^{-2}$ \\
			\hline
			$ba^{2}$ & $bab$ & $bab^{-1}$ & $b^{3}$ & $b^{2}a^{-1}$ & $ba^{-1}b$ & $ba^{-1}b^{-1}$ & $a^{2}b^{2}$ & $abab$ & $ab^{-1}ab$ & $ab^{-3}$ \\
			\noalign{\hrule height .03cm}
			\fbox{$V_{9}$} & $a$ & $b$ & $a^{-1}$ & $ab^{-1}$ & $a^{-2}$ & $a^{-1}b^{-1}$ & $b^{-1}a^{-1}$ & $a^{2}b^{-1}$ & $aba$ & $ab^{-2}$ \\
			\hline
			$bab$ & $bab^{-1}$ & $b^{2}a^{-1}$ & $ba^{-1}b$ & $ba^{-1}b^{-1}$ & $a^{-2}b^{-1}$ & $b^{-3}$ & $a^{2}b^{-2}$ & $baba$ & $bab^{-2}$ & $b^{-1}aba$ \\
			\noalign{\hrule height .04cm}
			\fbox{$U_{10}$} & $a$ & $a^{-1}$ & $b^{-1}$ & $ab^{-1}$ & $a^{-2}$ & $b^{-1}a$ & $b^{-1}a^{-1}$ & $a^{2}b$ & $ab^{2}$ & $ab^{-1}a$ \\
			\hline
			$ba^{2}$ & $bab$ & $b^{3}$ & $ba^{-1}b$ & $a^{-1}b^{2}$ & $b^{-1}ab$ & $b^{-1}a^{-1}b$ & $a^{2}b^{2}$ & $abab$ & $ab^{-1}ab$ & $a^{-1}b^{3}$ \\
			\noalign{\hrule height .03cm}
			\fbox{$V_{10}$} & $a$ & $b$ & $a^{-1}$ & $ab$ & $ba$ & $a^{-1}b$ & $a^{-2}$ & $a^{2}b^{-1}$ & $aba$ & $ab^{2}$ \\
			\hline
			$bab$ & $ba^{-1}b$ & $a^{-1}b^{2}$ & $a^{-2}b^{-1}$ & $b^{-1}ab$ & $b^{-1}a^{-1}b$ & $b^{-3}$ & $a^{2}b^{-2}$ & $baba$ & $b^{-1}aba$ & $b^{-1}a^{-1}b^{2}$ \\
			\noalign{\hrule height .04cm}
			\fbox{$U_{11}$} & $a$ & $a^{-1}$ & $b^{-1}$ & $ab$ & $ba^{-1}$ & $a^{-1}b$ & $a^{-2}$ & $a^{2}b$ & $ab^{2}$ & $ab^{-1}a$ \\
			\hline
			$ba^{2}$ & $bab$ & $b^{3}$ & $ba^{-1}b$ & $a^{-1}b^{2}$ & $b^{-1}ab$ & $b^{-1}a^{-1}b$ & $a^{2}b^{2}$ & $baba$ & $bab^{-1}a$ & $b^{-1}ab^{2}$ \\
			\noalign{\hrule height .03cm}
			\fbox{$V_{11}$} & $a$ & $b$ & $a^{-1}$ & $a^{-2}$ & $a^{-1}b^{-1}$ & $b^{-1}a$ & $b^{-1}a^{-1}$ & $a^{2}b^{-1}$ & $aba$ & $ab^{2}$ \\
			\hline
			$bab$ & $ba^{-1}b$ & $a^{-1}b^{2}$ & $a^{-2}b^{-1}$ & $b^{-1}ab$ & $b^{-1}a^{-1}b$ & $b^{-3}$ & $a^{2}b^{-2}$ & $abab$ & $abab^{-1}$ & $ab^{3}$ \\
			\noalign{\hrule height .04cm}
			\fbox{$U_{12}$} & $a$ & $a^{-1}$ & $b^{-1}$ & $a^{2}$ & $ab$ & $ba$ & $a^{-1}b$ & $a^{2}b$ & $ab^{2}$ & $ab^{-1}a$ \\
			\hline
			$ba^{2}$ & $bab$ & $b^{3}$ & $ba^{-1}b$ & $a^{-1}b^{2}$ & $b^{-1}ab$ & $b^{-1}a^{-1}b$ & $ba^{2}b$ & $baba$ & $bab^{-1}a$ & $b^{-1}a^{-1}b^{2}$ \\
			\noalign{\hrule height .03cm}
			\fbox{$V_{12}$} & $a$ & $b$ & $a^{-1}$ & $a^{2}$ & $ab^{-1}$ & $b^{-1}a$ & $b^{-1}a^{-1}$ & $a^{2}b^{-1}$ & $aba$ & $ab^{2}$ \\
			\hline
			$bab$ & $ba^{-1}b$ & $a^{-1}b^{2}$ & $a^{-2}b^{-1}$ & $b^{-1}ab$ & $b^{-1}a^{-1}b$ & $b^{-3}$ & $abab$ & $abab^{-1}$ & $a^{-1}b^{3}$ & $a^{-1}b^{2}a^{-1}$ \\
			\noalign{\hrule height .04cm}
			\fbox{$U_{13}$} & $a$ & $a^{-1}$ & $b^{-1}$ & $ab^{-1}$ & $a^{-2}$ & $a^{-1}b^{-1}$ & $b^{-1}a$ & $a^{2}b$ & $ab^{-1}a$ & $ab^{-2}$ \\
			\hline
			$ba^{2}$ & $bab$ & $bab^{-1}$ & $b^{3}$ & $b^{2}a^{-1}$ & $ba^{-1}b$ & $ba^{-1}b^{-1}$ & $a^{2}b^{2}$ & $baba$ & $bab^{-1}a$ & $b^{3}a^{-1}$ \\
			\noalign{\hrule height .03cm}
			\fbox{$V_{13}$} & $a$ & $b$ & $a^{-1}$ & $ab$ & $ba$ & $ba^{-1}$ & $a^{-2}$ & $a^{2}b^{-1}$ & $aba$ & $ab^{-2}$ \\
			\hline
			$bab$ & $bab^{-1}$ & $b^{2}a^{-1}$ & $ba^{-1}b$ & $ba^{-1}b^{-1}$ & $a^{-2}b^{-1}$ & $b^{-3}$ & $a^{2}b^{-2}$ & $abab$ & $abab^{-1}$ & $b^{2}a^{-1}b^{-1}$ \\
			\noalign{\hrule height .04cm}
			\fbox{$U_{14}$} & $a$ & $a^{-1}$ & $b^{-1}$ & $a^{2}$ & $ab^{-1}$ & $a^{-1}b^{-1}$ & $b^{-1}a^{-1}$ & $a^{2}b$ & $ab^{-1}a$ & $ab^{-2}$ \\
			\hline
			$ba^{2}$ & $bab$ & $bab^{-1}$ & $b^{3}$ & $b^{2}a^{-1}$ & $ba^{-1}b$ & $ba^{-1}b^{-1}$ & $ba^{2}b$ & $baba$ & $bab^{-1}a$ & $bab^{-2}$ \\
			\noalign{\hrule height .03cm}
			\fbox{$V_{14}$} & $a$ & $b$ & $a^{-1}$ & $a^{2}$ & $ba$ & $ba^{-1}$ & $a^{-1}b$ & $a^{2}b^{-1}$ & $aba$ & $ab^{-2}$ \\
			\hline
			$bab$ & $bab^{-1}$ & $b^{2}a^{-1}$ & $ba^{-1}b$ & $ba^{-1}b^{-1}$ & $a^{-2}b^{-1}$ & $b^{-3}$ & $abab$ & $abab^{-1}$ & $ab^{-3}$ & $a^{-1}b^{2}a^{-1}$ \\
			\noalign{\hrule height .04cm}
			\fbox{$U_{15}$} & $a$ & $a^{-1}$ & $b^{-1}$ & $a^{2}$ & $ab$ & $ba$ & $ba^{-1}$ & $a^{2}b$ & $ab^{-1}a$ & $ab^{-2}$ \\
			\hline
			$ba^{2}$ & $bab$ & $bab^{-1}$ & $b^{3}$ & $b^{2}a^{-1}$ & $ba^{-1}b$ & $ba^{-1}b^{-1}$ & $abab$ & $ab^{-1}ab$ & $ba^{2}b$ & $b^{2}a^{-1}b^{-1}$ \\
			\noalign{\hrule height .03cm}
			\fbox{$V_{15}$} & $a$ & $b$ & $a^{-1}$ & $a^{2}$ & $ab^{-1}$ & $a^{-1}b^{-1}$ & $b^{-1}a$ & $a^{2}b^{-1}$ & $aba$ & $ab^{-2}$ \\
			\hline
			$bab$ & $bab^{-1}$ & $b^{2}a^{-1}$ & $ba^{-1}b$ & $ba^{-1}b^{-1}$ & $a^{-2}b^{-1}$ & $b^{-3}$ & $baba$ & $b^{3}a^{-1}$ & $a^{-1}b^{2}a^{-1}$ & $b^{-1}aba$ \\
			\noalign{\hrule height .04cm}
			\fbox{$U_{16}$} & $b$ & $a^{-1}$ & $b^{-1}$ & $ab$ & $ab^{-1}$ & $ba$ & $b^{2}$ & $a^{3}$ & $a^{2}b$ & $a^{2}b^{-1}$ \\
			\hline
			$aba$ & $ab^{2}$ & $aba^{-1}$ & $ab^{-1}a$ & $ab^{-1}a^{-1}$ & $ab^{-2}$ & $ba^{-1}b$ & $a^{2}b^{-1}a^{-1}$ & $a^{2}b^{-2}$ & $baba$ & $ba^{-1}ba$ \\
			\noalign{\hrule height .03cm}
			\fbox{$V_{16}$} & $a$ & $b$ & $b^{-1}$ & $b^{2}$ & $ba^{-1}$ & $a^{-1}b$ & $b^{-1}a^{-1}$ & $a^{2}b$ & $a^{2}b^{-1}$ & $aba$ \\
			\hline
			$aba^{-1}$ & $ab^{-1}a$ & $ab^{-1}a^{-1}$ & $bab$ & $b^{2}a^{-1}$ & $a^{-1}b^{2}$ & $a^{-3}$ & $a^{3}b^{-1}$ & $abab$ & $a^{-1}bab$ & $a^{-1}b^{2}a^{-1}$ \\
			\noalign{\hrule height .04cm}
			\fbox{$U_{17}$} & $b$ & $a^{-1}$ & $b^{-1}$ & $ab$ & $ab^{-1}$ & $b^{-1}a$ & $b^{-2}$ & $a^{3}$ & $a^{2}b$ & $a^{2}b^{-1}$ \\
			\hline
			$aba$ & $ab^{2}$ & $aba^{-1}$ & $ab^{-1}a$ & $ab^{-1}a^{-1}$ & $ab^{-2}$ & $ba^{-1}b$ & $a^{2}b^{2}$ & $a^{2}ba^{-1}$ & $baba$ & $ba^{-1}ba$ \\
			\noalign{\hrule height .03cm}
			\fbox{$V_{17}$} & $a$ & $b$ & $b^{-1}$ & $ba^{-1}$ & $a^{-1}b^{-1}$ & $b^{-1}a^{-1}$ & $b^{-2}$ & $a^{2}b$ & $a^{2}b^{-1}$ & $aba$ \\
			\hline
			$aba^{-1}$ & $ab^{-1}a$ & $ab^{-1}a^{-1}$ & $bab$ & $b^{2}a^{-1}$ & $a^{-1}b^{2}$ & $a^{-3}$ & $a^{3}b$ & $abab$ & $ba^{2}b$ & $a^{-1}bab$ \\
			\noalign{\hrule height .04cm}
			\fbox{$U_{18}$} & $b$ & $a^{-1}$ & $b^{-1}$ & $b^{2}$ & $ba^{-1}$ & $a^{-1}b^{-1}$ & $b^{-1}a^{-1}$ & $a^{3}$ & $a^{2}b$ & $a^{2}b^{-1}$ \\
			\hline
			$aba$ & $ab^{2}$ & $aba^{-1}$ & $ab^{-1}a$ & $ab^{-1}a^{-1}$ & $ab^{-2}$ & $ba^{-1}b$ & $a^{3}b$ & $a^{2}b^{-2}$ & $abab$ & $aba^{-1}b$ \\
			\noalign{\hrule height .03cm}
			\fbox{$V_{18}$} & $a$ & $b$ & $b^{-1}$ & $ab$ & $ab^{-1}$ & $b^{2}$ & $b^{-1}a$ & $a^{2}b$ & $a^{2}b^{-1}$ & $aba$ \\
			\hline
			$aba^{-1}$ & $ab^{-1}a$ & $ab^{-1}a^{-1}$ & $bab$ & $b^{2}a^{-1}$ & $a^{-1}b^{2}$ & $a^{-3}$ & $a^{2}ba^{-1}$ & $baba$ & $baba^{-1}$ & $a^{-1}b^{2}a^{-1}$ \\
			\hline
		\end{tabular}
		\end{center}
	}
\end{table}


{\small
\begin{thebibliography}{99}
\bibitem{kissat}
A.\ Biere, K.\ Fazekas, M.\ Fleury, M.\ Heisinger.
CaDiCaL, Kissat, Paracooba, Plingeling and Treengeling entering the SAT Competition 2020.
In Proc.~of {SAT Competition} 2020 -- Solver and Benchmark Descriptions,
Department of Computer Science Report Series B., 2020, 50--53.

\bibitem{bowditch}
  B.\ Bowditch. A variation on the unique product property.
  J.\ London Math.\ Soc.\ 62(3), 2000, 813--826.


\bibitem{thomas}
  C.\ M.\ Campbell, R.\ M.\ Thomas.
  On infinite groups of Fibonacci type.
Proc.\ Edinburg Math.\ Soc.\ 29, 1986, 225--232.


\bibitem{C} 
W.\ Carter. 
New examples of torsion-free non-unique product groups. 
J.\ Group Theory, 17(3), 2014, 445--464.

\bibitem{cliff}
  G.\ H.\ Cliff.
  Zero divisors and idempotents in group rings.
Can.\ J.\ Math., Vol.\ XXXII, No.\ 3, 1980, 596--602.

\bibitem{CL}
W.\ Craig, P.\ Linnell. 
Unique product groups and congruence subgroups. 
J.\ Algebra Appl., 21(02):2250025, 2022.

\bibitem{sat2}
Cryptominisat.   
\url{https://github.com/msoos/cryptominisat}.

\bibitem{DPT}
  M.\ K.\ D\k{a}bkowsky, J.\ H.\ Przytycki, A.\ A.\ Togha.
  Non-Left-Orderable $3$-Manifold Groups.
Canad.\ Math.\ Bull.\ 48X, 2005, 32--40.

\bibitem{FisherSanchezPeralta2023}
S.~P.~Fisher, P.~Sánchez-Peralta,
\newblock Division rings for group algebras of virtually compact special groups and $3$-manifold groups,
\newblock \emph{arXiv preprint} arXiv:2303.08165 (2023).



\bibitem{fox}
  C.\ D.\ Fox. Can a Fibonacci group be a unique products group?
  Bull.\ Aust.\ Math.\ Soc.\ 19, 1978, 475--477.

\bibitem{gap} 
The GAP Group. 
GAP -- Groups, Algorithms and Programming. \url{https://gap-system.org} 
 
\bibitem{G}
G.\ Gardam.
A counterexample to the unit conjecture for group rings.
Annals of Mathematics 194(3), 2021, 967--979. 

\bibitem{G2}
G.\ Gardam.
Non-trivial units of complex group rings.
  Preprint \url{arxiv.org/abs/2312.05240}, 2023.
  
\bibitem{helling}
H.\ Helling, A.\ C.\  Kim, J.\ L.\  Mennicke.
A geometric study of Fibonacci groups. 
J.\ Lie Theory 8, 1998,  1--23.


\bibitem{H1}
G.\ Higman.
Units in group rings.
D.Phil thesis, University of Oxford, 1940.

\bibitem{handbook}
  D.\ F.\ Holt, E.\ A.\ O'Brien, B.\ Eick.
  Handbook of Computational Group Theory.
  CRC Press, 2005.


\bibitem{johnson}
  D.\ L.\ Johnson.
  Topics in the theory of group presentations. London Math.\ Soc.\ Lecture Notes Series 42.
  Cambridge University Press, 1980.
  

\bibitem{K1}
I.\ Kaplansky. 
``Problems in the theory of rings" revisited. 
Amer.\ Math.\ Monthly, 77, 1970, 445--454.

\bibitem{kim} 
  A.\  C.\ Kim, A.\ Vesnin. Fractional Fibonacci groups and manifolds.
 Sibirsk.\ Mat.\ Zh.\ 39, 1998, 765--775; translation in Siberian Math.\ J.\ 39, 1998, 655--664.


\bibitem{LS}
  R.\ C.\ Lyndon, P.\ E.\ Schupp.
 Combinatorial Group Theory, Springer, 1977.
 
\bibitem{murray}
  A.\ Murray.
  More Counterexamples to the Unit Conjecture for Group Rings
  Preprint \url{arxiv.org/abs/2106.02147}, 2021.

\bibitem{nielsen}
P.\ P.\ Nielsen, L.\ Soelberg.  
Small sets without unique products in torsion-free groups.
Journal Alg.\ Appl.\ 23(8), 2024, 2550050.
  
\bibitem{blog}
A.\ Nies. Logic Blog 2022. \url{https://arxiv.org/abs/2302.11853}
  
\bibitem{passman}  
D.\ S.\ Passman. The algebraic structure of group rings. Courier Corporation, 2011.
  
\bibitem{P}
S.\ Promislow. 
A simple example of a torsion-free, non unique product group.
Bull.\ Lond.\ Math.\ Soc.,20(4), 1988, 302–304.

\bibitem{RKD}
J.\ Raimbault, S.\ Kionke, N.\ Dunfield. 
On geometric aspects of diffuse groups.
Doc.\ Math., 21, 2016, 873--915.

\bibitem{RS}
E.\ Rips, Y.\ Segev. 
Torsion-free group without unique product property. 
J.\ Algebra, 108(1), 1987, 116--126.


\bibitem{rotman}
J.\ J.\ Rotman. Introduction to the theory of groups. Springer, 1995.  


\bibitem{Rudin.Schneider:64}
W.~Rudin, H.~Schneider.
\newblock Idempotents in group rings.
\newblock 1964.
\newblock US Department of the Army. Mathematics Research Center,
available at
\url{https://people.math.wisc.edu/hans/paper_archive/scanned_papers/hs020.pdf}.


\bibitem{segal}
D.\ Segal. Polycyclic Groups. Cambridge University Press, 1983.

\bibitem{SpallittaBiereSebastiani2025}
G.~Spallitta, A.~Biere, and R.~Sebastiani,
\newblock Disjoint projected enumeration for SAT and SMT without blocking clauses,
\newblock {\em Artificial Intelligence}, 2025.

\bibitem{TabularAllSATCode}
G.~Spallitta, A.~Biere, and R.~Sebastiani,
\newblock TabularAllSAT source code,
\newblock Zenodo, 2024.
\newblock DOI: 10.5281/zenodo.14197776.


\end{thebibliography}

}
	\end{document}